\UseRawInputEncoding
\documentclass[11pt,a4paper]{article}
\usepackage[T1]{fontenc}
\usepackage[utf8]{inputenc}
\usepackage{lmodern,microtype}
\usepackage[margin=25mm,headheight=15pt]{geometry}
\usepackage{amsmath,amssymb,amsthm,mathtools,bm}
\usepackage{enumitem,booktabs,array,fancyhdr}
\usepackage{xcolor}
\usepackage[colorlinks=true,linkcolor=blue!50!black,citecolor=blue!50!black,urlcolor=blue!50!black]{hyperref}
\hypersetup{pdftitle={The non-uniqueness of Leray solutions of the forced MHD equations},pdfauthor={Fei Xu, Yong Zhang}}
\numberwithin{equation}{section}
\theoremstyle{plain}
\newtheorem{theorem}{Theorem}[section]
\newtheorem{lemma}[theorem]{Lemma}
\newtheorem{proposition}[theorem]{Proposition}

\theoremstyle{definition}
\newtheorem{definition}[theorem]{Definition}
\theoremstyle{remark}
\newtheorem{remark}[theorem]{Remark}
\newcommand{\R}{\mathbb R}
\newcommand{\C}{\mathbb C}

\newcommand{\HH}{\mathcal H}
\newcommand{\PP}{\mathbb P}
\newcommand{\M}{\mathcal M}
\newcommand{\N}{\mathcal N}

\newcommand{\B}{\mathcal B}
\newcommand{\Rea}{\operatorname{Re}}
\newcommand{\supp}{\operatorname{supp}}
\newcommand{\Ran}{\operatorname{Ran}}
\newcommand{\diver}{\operatorname{div}}
\newcommand{\ess}{\mathrm{ess}}

\newcommand{\dd}{\,\mathrm d}
\newcommand{\norm}[1]{\left\|#1\right\|}
\newcommand{\ip}[2]{\left\langle #1,#2\right\rangle}
\begin{document}
\fancyhf{}\fancyhead[L]{\small Non-uniqueness for the forced MHD equations}
\fancyhead[R]{}\fancyfoot[C]{\thepage}
\title{\bfseries The non-uniqueness of Leray solutions of the forced MHD equations}
\author{Fei Xu\thanks{School of Mathematical Sciences, Jiangsu University, Zhenjiang, Jiangsu 212013, P.R. China. Email: \texttt{xufeiujs@126.com}.}
\and Yong Zhang\thanks{School of Mathematical Sciences, Jiangsu University, Zhenjiang, Jiangsu 212013, P.R. China. Email: \texttt{zhangyong@ujs.edu.cn}.}}
\date{}

\maketitle
\begin{abstract}
We construct two distinct Leray solutions of the three-dimensional forced viscous, resistive magnetohydrodynamic equations in the whole space, with zero initial data and identical forcing. Our construction begins with a periodic axisymmetric unstable MHD mode associated with a rotating column and a nonconstant axial magnetic field. The rotating profile is also centrifugally unstable at zero magnetic field; the construction uses an unstable MHD eigenvalue and does not require magnetic destabilization of a hydrodynamically stable flow. Our spectral construction tracks a prescribed unstable MHD eigenvalue under radial truncation, the addition of small positive viscosity, and toroidal closure. The first two steps use a decomposition into a skew-adjoint principal part and compact feedback at a fixed Fourier frequency. Keeping the viscosity fixed, we bend the column into a compactly supported toroidal profile and obtain a genuine whole-space unstable eigenvalue through convergence of localized Stokes resolvents. A small-dilation limit and an exact rescaling then yield an unstable background in similarity variables. A nonlinear trajectory along its unstable direction produces the second solution. Both solutions are smooth for positive times and satisfy the energy equality. The forcing acts in both the velocity and magnetic equations.
\end{abstract}
\noindent\textbf{Keywords:} Magnetohydrodynamics; Leray solutions; non-uniqueness; magneto-rotational instability; spectral transfer.\\
\noindent\textbf{2020 Mathematics Subject Classification:} 35Q35, 35B35, 76W05, 47D06.

\section{Introduction and main results}

The three-dimensional incompressible viscous, resistive magnetohydrodynamic equations describe the coupled evolution of the velocity and magnetic field in an electrically conducting fluid. We consider the forced system with unit viscosity and resistivity:
\begin{equation}\label{eq:mhd}
\begin{cases}
\partial_t u+u\cdot\nabla u-b\cdot\nabla b-\Delta u+\nabla p=f,\\
\partial_t b+u\cdot\nabla b-b\cdot\nabla u-\Delta b=g,\\
\diver u=\diver b=0,
\end{cases}\qquad (x,t)\in\R^3\times(0,T).
\end{equation}
Here $p$ denotes the total pressure. For finite-energy initial data, energy estimates and compactness yield global weak solutions; see the classical works of Duvaut--Lions and Sermange--Temam~\cite{DL,ST}. The natural regularity class is
\begin{equation}\label{eq:energyclass}
(u,b)\in L^\infty(0,T;L^2_\sigma(\R^3)^2)\cap L^2(0,T;\dot H^1(\R^3)^2).
\end{equation}
Our purpose is to construct distinct solutions in this energy class with the same initial data and forcing.

Recent non-uniqueness results for the incompressible Navier--Stokes and MHD equations concern several distinct solution classes. Finite kinetic energy, continuity of the energy, a prescribed decreasing energy profile, and the Leray--Hopf energy inequality are different requirements. The last includes the time integral of viscous and, for MHD, resistive dissipation. We keep these distinctions explicit when comparing the constructions below.

For the unforced three-dimensional Navier--Stokes equations on the torus, Buckmaster--Vicol~\cite{BV} introduced intermittent convex integration to construct non-unique finite-energy distributional solutions. Buckmaster--Colombo--Vicol~\cite{BCV} obtained weak solutions smooth outside a fractal set of times of dimension below one. Cheskidov--Luo~\cite{CL22,CL23} established sharp temporal non-uniqueness below the $L^2_tL^\infty_x$ endpoint and nearly $L^2$-critical examples in two dimensions. Subsequent refinements include non-uniqueness in $C_tL^q_x$ for exponents slightly above two~\cite{MZ} and continuous-energy solutions attaining every prescribed finite-energy initial datum~\cite{CZZNS}. Miao--Nie--Ye developed whole-space constructions both in $C_tL^2_x$ and in $L^p_tL^\infty_x$, $p<2$~\cite{MNYNS,MNYSharp}. These results do not assert ordinary-viscosity Leray--Hopf non-uniqueness.

Fractional dissipation provides further comparisons. Colombo--De Lellis--De Rosa~\cite{CDLDR} and De Rosa~\cite{DeRosa} proved genuine Leray--Hopf non-uniqueness for $(-\Delta)^\alpha$ below $\alpha=1/5$ and $\alpha=1/3$, respectively; Gorini established density of wild initial data in the latter regime~\cite{Gorini}. For this convention for the dissipation, the Navier--Stokes Lions energy-class threshold in dimension $d$ is $\alpha_L(d)=(d+2)/4$: it equals $1$ in two dimensions and $5/4$ in three dimensions. Luo--Titi~\cite{LT} constructed three-dimensional non-unique weak solutions below $\alpha_L(3)$, while Luo--Qu treated two-dimensional hypoviscous equations~\cite{LQ}. Li--Qu--Zeng--Zhang~\cite{LQZZNS} obtained sharpness results near generalized Lady\v{z}enskaja--Prodi--Serrin (LPS) endpoints, including three-dimensional hyperdissipation with $5/4\le\alpha<2$. Du--Li~\cite{DuLi} treated the two-dimensional range $1\le\alpha<3/2$. The non-uniqueness in these last two works occurs in LPS-supercritical solution classes; it does not contradict uniqueness in the corresponding Leray--Hopf class at or above the dimension-dependent Lions threshold.

A different route uses instability in similarity variables. Jia--\v{S}ver\'ak established existence of large forward self-similar solutions~\cite{JS14} and formulated spectral and bifurcation conditions implying non-uniqueness after localization of the initial data~\cite{JS}. Guillod--\v{S}ver\'ak supplied numerical evidence for the proposed symmetry-breaking mechanism~\cite{GS}. Albritton--Bru\'e--Colombo~\cite{ABC} obtained an unconditional analytic construction for the forced three-dimensional equations with zero initial data and common force in $L^1_tL^2_x$; their gluing argument extends the conclusion to smooth bounded domains and the torus~\cite{ABCglue}. Albritton--Colombo treated the forced two-dimensional equations with dissipation $(-\Delta)^{\beta/2}$, $0<\beta<2$~\cite{AC}, and Khor--Miao--Su covered the three-dimensional forced fractional problem for $1/2<\alpha<5/4$~\cite{KMS}. Neither range includes its upper endpoint. Hofmanov\'a--Zhu--Zhu also proved that forces producing deterministic Leray--Hopf non-uniqueness from zero initial data on $\R^3$ are dense in $L^1_tL^2_x$, and that every divergence-free finite-energy initial datum admits a force producing non-uniqueness~\cite{HZZleray}.

Recent critical-data constructions require separate interpretation. Coiculescu--Palasek proved non-uniqueness of positive-time smooth solutions from critical $\mathrm{BMO}^{-1}$ data~\cite{CP}. Cheskidov--Dai--Palasek developed an inverse-energy-cascade branching mechanism with instantaneous Type I singular times~\cite{CDP}, and Cheskidov--Hou studied singular mild solutions in negative Besov spaces, with a paraproduct interpretation of the nonlinearity~\cite{CH}. These are not finite-energy Leray--Hopf conclusions. For the unforced three-dimensional whole-space equations, a preprint of Hou--Wang--Yang announces a computer-assisted proof of Leray--Hopf non-uniqueness from nonzero, singular, compactly supported initial data~\cite{HWY}. In two dimensions, Albritton--Guillod--Korobkov--Ren establish large self-similar existence and present numerical non-uniqueness evidence outside the energy class~\cite{AGKR}. These numerical, conditional, computer-assisted, and analytic results have different hypotheses.

For ideal MHD, early symmetry-reduced examples are related to the passive-tracer construction of Bronzi--Lopes Filho--Nussenzveig Lopes~\cite{BLFL}. Beekie--Buckmaster--Vicol constructed weak solutions with nonconserved magnetic helicity~\cite{BBV}. Faraco--Lindberg--Sz\'ekelyhidi obtained bounded solutions compactly supported in space-time~\cite{FLS21} and developed relaxation results with energy and cross-helicity dissipation and prescribed constant magnetic helicity~\cite{FLS24}. More recent work concerns Onsager-type H\"older flexibility~\cite{MNYideal} and H\"older solutions with nonzero magnetic helicity~\cite{EPT,Giardi}. These ideal constructions address conservation laws and do not supply the positive-viscosity energy solutions considered here.

For viscous and resistive MHD, Li--Zeng--Zhang established weak non-uniqueness by intermittent convex integration below the Lions exponent~\cite{LZZ22}, and subsequently in LPS-supercritical classes extending beyond that exponent~\cite{LZZ24}. Miao--Ye constructed $C_tL^2_x$ weak solutions with controllable total energy and cross helicity~\cite{MY}; Nie--Ye proved strong non-uniqueness in $L^p_tL^\infty_x$, $1\le p<2$~\cite{NY}. Cheskidov--Zeng--Zhang constructed global continuous-energy weak solutions and decreasing-energy examples for the Navier--Stokes and MHD equations~\cite{CZZMHD}. Here energy monotonicity is distinct from the Leray--Hopf inequality with its integrated dissipation. Recent preprints also treat whole-space two-dimensional classical MHD in $L^2_tL^p_x\cap L^1_tW^{1,p}_x$, $1\le p<\infty$~\cite{MNY2D}, and instantaneous-singularity branching for MHD solutions smooth away from a prescribed singular time~\cite{Dai}. The former classes lie beyond the two-dimensional energy uniqueness theory; smoothness away from an instantaneous singular time likewise does not provide a Leray energy inequality through that time.

Stochastic non-uniqueness has developed in parallel. Hofmanov\'a--Zhu--Zhu established non-uniqueness in law for analytically weak Navier--Stokes solutions~\cite{HZZlaw} and non-unique paracontrolled solutions under singular space-time white noise~\cite{HZZwhite}. Their forced Leray construction treats linear multiplicative noise~\cite{HZZleray}, whereas Bru\'e--Jin--Li--Zhang treat additive noise~\cite{BJLZ}; in these forced results the additional body force may be adapted to the noise. For stochastic MHD, Yamazaki obtained non-uniqueness in law under hypodissipation~\cite{Yamnoise}, and later analytically weak non-uniqueness from prescribed data for diffusion ranging from the full Laplacian to below the Lions exponent~\cite{YamLions}. The latter MHD solutions are not claimed to belong to the Leray energy class.

Our construction follows the similarity-instability strategy of \cite{ABC}. Its additional spectral result tracks a prescribed unstable mode of a rotating MHD column through radial truncation, the addition of positive viscosity, and toroidal closure. This produces a smooth, compactly supported whole-space background with a genuine unstable eigenvalue while preserving both divergence constraints and the nonlocal pressure projection. The similarity-variable and nonlinear arguments then yield two Leray solutions with common zero initial data and common forcing, for classical viscosity and resistivity, satisfying the full energy equality.

For the existence statement alone, a perturbative construction from the unstable Navier--Stokes background of \cite{ABC} is also available. At zero magnetic field the similarity linearization $L_0$ is block diagonal, with velocity and induction blocks on the common dilation--Laplacian domain. Both blocks are relatively compact perturbations of the free generator, so the unstable velocity eigenvalue remains isolated in the product space. Adding a field $\varepsilon B$, with $0\ne B$ smooth, compactly supported and divergence-free, changes this operator by $\varepsilon V_B$, where $V_B$ is bounded from the common graph domain into the energy space. On an isolating circle, $\varepsilon V_B(\lambda-L_0)^{-1}$ therefore tends to zero in norm, and the unstable Riesz projection persists. Taking the supports of $B$ and the velocity background to be disjoint gives magnetic force $G=-\varepsilon AB\ne0$, with $A=\Delta+\frac12(1+x\cdot\nabla)$, since $\langle-AB,B\rangle=\|\nabla B\|_2^2+\frac14\|B\|_2^2>0$. The construction developed below additionally follows a specified periodic MHD eigenvalue through localization and changing geometry.

Our starting point is the rigorous analysis of magneto-rotational instability (MRI) by Lin--Wang--Wu~\cite{LWW}. They consider a rotating velocity field $r\Omega(r)e_\theta$ and an axial magnetic field $\varepsilon b_0(r)e_z$, and establish unstable eigenvalues and the Hamiltonian spectral structure in an axially periodic domain. Their corollary implies that, if $(\Omega^2)'<0$ somewhere, a sufficiently weak positive magnetic field satisfying their hypotheses gives rise to linear instability. We invoke this periodic instability result directly, with the precise version and endpoint hypotheses checked in Section 2. In this paper, the MRI input refers to that axisymmetric MHD spectral criterion. The compact rotating profile chosen below also has a Rayleigh-unstable region in the zero-field problem, so we do not identify the unstable eigenvalue as a branch created by the magnetic field. The loss of a positive lower bound for the magnetic field after truncation, the passage from the periodic axial direction to the whole space, and the finite energy of the background require separate arguments.

Truncating a periodic eigenfunction does not resolve these issues. As long as translation invariance along the column is retained, a periodic mode does not belong to the whole-space $L^2$ class; an axial cutoff, on the other hand, couples different Fourier frequencies. Moreover, the magnetic field must remain divergence free, and the whole-space pressure projection is nonlocal. We first carry out radial localization and spectral transfer to small positive viscosity at a fixed axial frequency. We then keep the viscosity strictly positive and close the axial direction into a large circle. During bending, we compare localized Stokes resolvents on the full transverse space, so the original radial perturbation subspace need not remain invariant.

We begin by specifying the notion of weak solution used below.
\begin{definition}\label{def:leray}
Let $(u_0,b_0)\in L^2_\sigma(\R^3)^2$ and $(f,g)\in L^1(0,T;L^2(\R^3)^2)$. A real-valued pair $(u,b)$ is a Leray energy solution if it belongs to~\eqref{eq:energyclass} and $C_w([0,T];L^2_\sigma{}^2)$, solves~\eqref{eq:mhd} in the sense of distributions, attains the initial data, and satisfies, for $0\le t\le T$,
\begin{align}\label{eq:leray}
\frac12\bigl(\norm{u(t)}_2^2+\norm{b(t)}_2^2\bigr)
+\int_0^t\bigl(\norm{\nabla u(s)}_2^2+\norm{\nabla b(s)}_2^2\bigr)\dd s
\le{}&\frac12\bigl(\norm{u_0}_2^2+\norm{b_0}_2^2\bigr)\notag\\
&+\int_0^t\!\int_{\R^3}(f\cdot u+g\cdot b)\dd x\dd s.
\end{align}
The solutions constructed here satisfy the energy equality and hence also the energy inequality starting from any intermediate time.
\end{definition}

Our main result is the following existence theorem. In Section 2, we realize the required MRI input with a concrete family of nonconstant magnetic profiles.
\begin{theorem}\label{thm:main}
There exist $T>0$, real-valued smooth, compactly supported forcing profiles $F,G$, and two distinct real-valued solutions $(u_j,b_j)$, $j=1,2$, of~\eqref{eq:mhd} with the same zero initial data and common forcing
\begin{equation}\label{eq:force}
(f,g)(x,t)=t^{-3/2}(F,G)(x/\sqrt t)\in L^1(0,T;L^2(\R^3)^2).
\end{equation}
Both solutions belong to
\[
C([0,T];L^2_\sigma(\R^3)^2)\cap L^2(0,T;\dot H^1(\R^3)^2),
\]
are smooth for $t>0$, and satisfy~\eqref{eq:leray} with equality. The background magnetic field and $G$ are both nonzero.
\end{theorem}

To explain the structure of the proof, we state two intermediate results. Let $P$ denote the whole-space Leray projection and set $\PP=\operatorname{diag}(P,P)$. Write
\begin{equation}\label{eq:N}
\N_Q(v,h)=-\begin{pmatrix}
U\cdot\nabla v+v\cdot\nabla U-B\cdot\nabla h-h\cdot\nabla B\\
U\cdot\nabla h+v\cdot\nabla B-B\cdot\nabla v-h\cdot\nabla U
\end{pmatrix},\qquad \M_Q=\PP\N_Q,
\end{equation}
where $Q=(U,B)$ is a divergence-free background, and let $\HH=L^2_\sigma(\R^3;\C^3)^2$.

The first result transfers periodic instability to a genuine eigenvalue associated with a finite-energy background.
\begin{theorem}\label{thm:physical}
There exist a real-valued pair $Q=(U,B)\in C_c^\infty(\R^3;\R^3)^2$, with $\diver U=\diver B=0$ and $B\not\equiv0$, and a viscosity $\nu>0$, such that
\[
T_Q=\nu\Delta+\M_Q,\qquad D(T_Q)=H^2(\R^3;\C^6)\cap\HH,
\]
has an isolated eigenvalue of finite algebraic multiplicity in $\{\Rea\lambda>0\}$. The background can be constructed from an MRI column with a nonconstant axial magnetic field by radial truncation and toroidal bending.
\end{theorem}

The second result gives the spectral information in similarity variables needed for the nonlinear construction.
\begin{theorem}\label{thm:ss}
There exists a real-valued smooth, compactly supported, divergence-free background $\overline Q=(\overline U,\overline B)$ such that
\begin{equation}\label{eq:Lss}
L=\Delta+\tfrac12(1+x\cdot\nabla)+\M_{\overline Q},\qquad
D(L)=\{w\in H^2\cap\HH:x\cdot\nabla w\in L^2\}
\end{equation}
has a discrete eigenvalue with positive real part. Its spectral bound $a=s(L)>0$ is attained at an eigenvalue $\lambda=a+i\omega$. For every $\epsilon>0$ and all real $s_2\ge s_1\ge0$, there is a constant such that
\begin{equation}\label{eq:main-smoothing}
\norm{e^{tL}}_{H^{s_1}\cap\HH\to H^{s_2}\cap\HH}
\le C_{s_1,s_2,\epsilon}(1+t^{-(s_2-s_1)/2})e^{(a+\epsilon)t},\qquad t>0.
\end{equation}
\end{theorem}

Theorem~\ref{thm:physical} is proved in Sections 2 and 3. In Section 4, we deduce Theorem~\ref{thm:ss}; in Section 5, we construct the unstable trajectory and prove Theorem~\ref{thm:main}. This separates the known periodic MRI result from the whole-space localization required here. In particular, unstable eigenvalues in similarity variables may be complex: only positivity of their real part is needed for non-uniqueness.

\section{Radial truncation and viscous spectral persistence at fixed frequency}
\subsection{The periodic MRI input and physical variables}

In cylindrical coordinates $(r,\theta,z)$, consider
\begin{equation}\label{eq:column}
U_\infty=r\Omega(r)e_\theta,\qquad B_\infty=b(r)e_z.
\end{equation}
Throughout this section, we assume that $\Omega\in C_c^\infty((r_-,r_+))$, where $0<r_-<r_+<\infty$, and that $b$ is a real-valued smooth function such that $b(|y|)$ is smooth on $\R^2$ and
\begin{equation}\label{eq:bclass}
\norm b_\infty+\norm{b'}_\infty<\infty,\qquad
\lim_{r\downarrow0}b'(r)=\lim_{r\to\infty}b'(r)=0.
\end{equation}
This background is an equilibrium of ideal MHD, with total pressure satisfying $p'(r)=r\Omega(r)^2$. Fix a nonzero integer $k$ and consider perturbations whose cylindrical components are independent of $\theta$ and have axial dependence $e^{ikz}$. Set
\[
X=L^2((0,\infty),r\dd r;\C^3),\qquad
H_k=\{v\in X:r^{-1}(rv_r)'+ikv_z=0\},\qquad\HH_k=H_k\times H_k,
\]
and let $P_k$ denote the orthogonal projection from $X$ onto $H_k$. At the axis, we require only regularity in Cartesian coordinates; the axis is not a physical boundary.

All divergence identities in the definition of $H_k$ are understood distributionally after identifying the cylindrical components with a Cartesian field on the transverse plane. In particular, $H_k$ is a closed subspace of $X$. The projection $P_k$ has norm one and leaves the azimuthal component unchanged. More explicitly, if $f\in X$, let $q$ solve
\begin{equation}\label{eq:s2-pressure-projection}
\int_0^\infty(q'\overline{\varphi'}+k^2q\overline\varphi)r\dd r
=\int_0^\infty(f_r\overline{\varphi'}-ik f_z\overline\varphi)r\dd r.
\end{equation}
The test space is the radial $H^1$ space with norm given by the left-hand quadratic form. Since $k\ne0$, this form is coercive, and the Riesz representation theorem gives a unique $q$ in that space. Then $P_kf=f-(q',0,ikq)$. This description imposes regularity through the whole transverse space rather than introducing a boundary condition at $r=0$. It also explains why a pressure gradient can be removed in the fixed-frequency energy space even when the pressure is not compactly supported.

Direct linearization yields
\begin{equation}\label{eq:components}
\begin{cases}
\partial_t v_r=2\Omega v_\theta+ikb h_r-q',\\
\partial_t v_\theta=-(2\Omega+r\Omega')v_r+ikb h_\theta,\\
\partial_t v_z=ikb h_z+b'h_r-ikq,\\
\partial_t h_r=ikb v_r,\\
\partial_t h_\theta=ikb v_\theta+r\Omega' h_r,\\
\partial_t h_z=ikb v_z-b'v_r.
\end{cases}
\end{equation}
Both divergence constraints are preserved; for example, $\diver_k(ikbv-b'v_re_z)=ikb'v_r-ikb'v_r=0$. The generator at fixed frequency is therefore the bounded operator on $\HH_k$ given by
\begin{align}\label{eq:SK}
M_{k,b}&=S_{k,b}+K_{k,b},\notag\\
S_{k,b}(v,h)&=(ikP_k(bh),ikP_k(bv)),\notag\\
(K_{k,b}(v,h))_1&=P_k\bigl(2\Omega v_\theta e_r-(2\Omega+r\Omega')v_re_\theta+b'h_re_z\bigr),\notag\\
(K_{k,b}(v,h))_2&=P_k\bigl(r\Omega'h_re_\theta-b'v_re_z\bigr).
\end{align}
The projection in the second component may be inserted because the full right-hand side of the magnetic equation in \eqref{eq:components} is already divergence-free.

Our spectral input is the existence of a nonconstant $b$ satisfying \eqref{eq:bclass} and a frequency $k\ne0$ for which $M_{k,b}$ has an eigenvalue with $\Rea\lambda_*>0$. This input can indeed be realized. Take
\begin{equation}\label{eq:example}
\Omega(r)=\begin{cases}\exp\bigl(-1/(1-(r-2)^2)\bigr),&1<r<3,\\0,&\text{otherwise},\end{cases}
\qquad b(r)=\varepsilon(1+d e^{-r^2}),\quad 0<d<\tfrac12.
\end{equation}
Then $(\Omega^2)'<0$ on $2<r<3$, while $b/\varepsilon$ has a positive lower bound and satisfies the endpoint conditions of \cite[Theorem 1.1]{LWW}. By \cite[Corollary 1.1(ii)]{LWW}, at least one nonzero integer mode has a positive real eigenvalue when $\varepsilon>0$ is sufficiently small.

For completeness, the endpoint conditions of that theorem hold with $\beta=2$ at the axis: $(\Omega^2)'$ vanishes near zero and $(b/\varepsilon)'=O(r)$. At infinity, $(\Omega^2)'$ vanishes and $(b/\varepsilon)'$ decays exponentially, so the required bounds hold for any fixed $\alpha>0$. The lower bound $b/\varepsilon\ge1$ is uniform in $r$. These are hypotheses on the untruncated seed only; a lower bound independent of $R$ is not used after radial localization. The cited eigenvalue is a point eigenvalue in the periodic energy space, and Fourier decomposition in the periodic coordinate supplies a nonzero integer frequency carrying a nonzero component of its eigenvector.

To identify the corresponding energy spaces, introduce the magnetic potential $\phi=rh_r/(ik)$. Thus $h_r=ik\phi/r$ and $h_z=-\phi'/r$, and the weighted energy of the magnetic potential is precisely the $L^2$ energy of the poloidal magnetic components. The additional velocity variable in the Hamiltonian formulation is $v_\theta+(\Omega'/b)\phi$. Since $\Omega'$ is supported in a fixed annulus on which $b$ has a positive lower bound, this change of variables and its inverse are bounded on the energy space. Hence the eigenmode supplied by the cited result yields a nonzero eigenvector of \eqref{eq:components} in $\HH_k$. We invoke the initial MRI result directly, without reproving the full Hamiltonian spectral theorem.

\begin{remark}[Scope of the seed instability]\label{rem:s2-rayleigh-scope}
The example above is not a hydrodynamically stable rotating flow. Indeed, define
\begin{equation}\label{eq:s2-rayleigh}
\kappa^2(r)=\frac1{r^3}(r^4\Omega^2)'=4\Omega^2+2r\Omega\Omega'.
\end{equation}
For \eqref{eq:example}, $\Omega'/\Omega=-2(r-2)/(1-(r-2)^2)^2$ on $(1,3)$, and hence
\[
\kappa^2(5/2)=-\frac{44}{9}\Omega(5/2)^2<0.
\]
Thus the zero-field column already satisfies the Rayleigh instability condition; compare \cite[Section 5]{LWW}. This feature is unavoidable for a nonzero compactly supported angular velocity: the nonnegative function $r^4\Omega^2$ is nonzero and returns to zero, so its derivative is negative somewhere. One can also check zero-field instability directly from the positive-$\mu$ pencil in the appendix by setting $b=0$, in which case its quadratic form becomes
\[
\mu\|w\|_{V_k}^2+\int_0^\infty\kappa^2(r)|w|^2r\dd r.
\]
A test function supported where $\kappa^2<0$ gives a negative value for small $\mu$, and the same compact spectral-crossing argument gives a positive eigenvalue. This observation does not identify which MHD eigenvalue continues which zero-field mode. The input used in the proof is the existence of an unstable eigenvalue for a nonzero, nonconstant magnetic field, with no assertion that this eigenvalue is created by the magnetic coupling. The contribution needed for non-uniqueness is its transfer to a finite-energy background and then to similarity variables.
\end{remark}

\subsection{Compact feedback and radial truncation}

At a fixed nonzero frequency, the divergence constraint yields the following local compactness property.
\begin{lemma}\label{lem:radcompact}
Let $c\in C_0((0,\infty))$, that is, let $c$ be continuous and vanish at both endpoints. Then
\[
J_cv=cv_r:H_k\longrightarrow L^2((0,\infty),r\dd r)
\]
is compact, as is its adjoint $J_c^*f=P_k(\overline c f e_r)$.
\end{lemma}
\begin{proof}
First suppose that $c$ is supported in an annulus $I\Subset(0,\infty)$. The identity $v_r'=-r^{-1}v_r-ikv_z$ gives
$\norm{v_r}_{H^1(I)}\le C_{I,k}\norm v_X$. Weighted and unweighted norms are equivalent on $I$, so compactness follows from the one-dimensional Rellich theorem. For general $c$, approximate $c$ uniformly by functions compactly supported in annuli. The corresponding operator norm error is bounded by the uniform approximation error. The result follows because compact operators are closed in operator norm and the adjoint of a compact operator is compact.

The restriction to the radial component is essential here. The divergence constraint gives a derivative bound for $v_r$ on each annulus, whereas an arbitrary bounded sequence of azimuthal components can oscillate without such a bound. Compactness of the pressure-projected term with radial output follows through the adjoint $J_c^*$, rather than through compactness of multiplication on all three components. This is the reason for the particular splitting in \eqref{eq:SK}.

\end{proof}

We can now justify the decomposition \eqref{eq:SK}.
\begin{proposition}\label{prop:skcompact}
The operator $S_{k,b}$ is bounded and skew-adjoint, $K_{k,b}$ is compact, and
\[
\norm{(z-S_{k,b})^{-1}}\le(\Rea z)^{-1}\quad(\Rea z>0).
\]
The spectrum of $M_{k,b}$ in the right half-plane consists entirely of isolated eigenvalues of finite algebraic multiplicity.
\end{proposition}
\begin{proof}
The compression $P_kb|_{H_k}$ of multiplication by the real-valued function $b$ is self-adjoint, so the off-diagonal blocks in \eqref{eq:SK}, multiplied by $ik$, define a skew-adjoint operator. The terms in $K_{k,b}$ involving $v_r$ or $h_r$ are compact by Lemma~\ref{lem:radcompact}; the required vanishing of $b'$ at the endpoints follows from \eqref{eq:bclass}. The remaining term $P_k(2\Omega v_\theta e_r)=2J_\Omega^*(v_\theta)$ is also compact. Finally, apply the analytic Fredholm theorem to
\[
I-(z-S_{k,b})^{-1}K_{k,b}
\]
This family is invertible for sufficiently large positive real $z$ by a Neumann series. Its inverse is therefore meromorphic in the right half-plane, with finite-rank principal parts. We refer to \cite{Kato} for the analytic Fredholm theorem and the Riesz projection theory used here and below.
\end{proof}

We record a compact-feedback lemma that applies both to radial truncation and to the limit of vanishing viscosity.
\begin{lemma}\label{lem:transfer}
Let $A_j,A_0$ be closed, densely defined operators on the same Hilbert space, with resolvents $G_j(z)$ on $O=\{\Rea z>\mu\}$. Assume that
\[
\norm{G_j(z)}\le(\Rea z-\mu)^{-1},\qquad G_j(z)x\to G_0(z)x
\]
locally uniformly for each $x$. Suppose that $K_j$ is compact and $\norm{K_j-K_0}\to0$. If $T_0=A_0+K_0$ has an isolated eigenvalue $\lambda_0$ in $O$, then every sufficiently small isolating disk contains an eigenvalue of $T_j=A_j+K_j$ for all sufficiently large $j$.
\end{lemma}
\begin{proof}
On any compact set $\Lambda\Subset O$, approximate $K_0$ by finite-rank operators and use strong convergence on their finite-dimensional ranges to obtain
\begin{equation}\label{eq:compact-conv}
\sup_{z\in\Lambda}\norm{G_j(z)K_j-G_0(z)K_0}\longrightarrow0.
\end{equation}
Set $F_j=I-G_jK_j$. These families are invertible at sufficiently large positive real parameters, so the analytic Fredholm theorem applies. On an isolating circle $\Gamma$, \eqref{eq:compact-conv} implies that $F_j^{-1}\to F_0^{-1}$ uniformly in operator norm. Moreover,
\[
(z-T_j)^{-1}=F_j(z)^{-1}G_j(z).
\]
This identity remains valid for unbounded $A_j$: if $u=F_j^{-1}G_jf$, then $u=G_jf+G_jK_ju\in D(A_j)$. The resolvents converge strongly, uniformly along $\Gamma$ for each input, and hence the Riesz projections satisfy $\Pi_jx\to\Pi_0x$. Choosing $0\ne x\in\Ran\Pi_0$ shows that $\Pi_j\ne0$. Since Fredholm theory ensures that the spectrum in the disk is discrete, $T_j$ has an eigenvalue there. We do not infer equality of the ranks of the projections from strong convergence.

To make the final implication explicit, take $x\in\Ran\Pi_0$ with $\|x\|=1$. Then $\|\Pi_jx-x\|\to0$, so $\|\Pi_jx\|\ge1/2$ for large $j$. If the interior of $\Gamma$ contained no spectrum of $T_j$, its resolvent would be analytic there and its contour integral would vanish. Thus the nonzero projection forces spectrum inside the circle. Discreteness in $O$ then makes this spectrum point spectrum. The argument preserves the existence of an eigenvalue in the disk; no convergence of normalized eigenvectors or exact multiplicity statement is required.

\end{proof}

Choose $\chi\in C_c^\infty([0,\infty))$ with $0\le\chi\le1$, $\chi=1$ on $[0,1]$, and $\chi=0$ on $[2,\infty)$, and set
\begin{equation}\label{eq:radcut}
b_R(r)=\chi(r/R)b(r),\qquad R>r_+.
\end{equation}
The axial magnetic field $b_Re_z$ remains divergence-free.
\begin{theorem}\label{thm:radial}
Suppose that $M_{k,b}$ has an eigenvalue with $\Rea\lambda_*>0$. Then every sufficiently small isolating disk contains an eigenvalue of $M_{k,b_R}$ for all sufficiently large $R$.
\end{theorem}
\begin{proof}
Dominated convergence yields $S_{k,b_R}\to S_{k,b}$ strongly. These principal parts are skew-adjoint and uniformly bounded in norm. The resolvent identity therefore gives strong resolvent convergence in the right half-plane, locally uniformly in the spectral parameter. The difference of the compact parts is controlled by
\[
b_R'-b'=(\chi(r/R)-1)b'+R^{-1}\chi'(r/R)b
\]
and hence
\begin{equation}\label{eq:Kdiff}
\norm{K_{k,b_R}-K_{k,b}}\le C\left(\sup_{r\ge R}|b'(r)|+R^{-1}\norm b_\infty\right)\longrightarrow0.
\end{equation}
The conclusion follows from Lemma~\ref{lem:transfer} with $\mu=0$.
\end{proof}

\begin{remark}\label{rem:method}
For the persistence of a prescribed eigenvalue considered here, the operator decomposition is more direct than a new application of the full Hamiltonian framework. The truncated field $b_R$ vanishes at large radii and thus no longer satisfies the positive lower bound assumed in \cite{LWW}; the factor $1/b_R$ in the Hamiltonian variables cannot be used directly either. The preceding proof requires only strong convergence of the skew-adjoint principal parts and operator norm convergence of the compact feedback. For the concrete profile \eqref{eq:example}, take a unit input $(0,h_\theta e_\theta)$ supported in $\{r>2R\}$. The fields $h_\theta e_\theta$ and $b h_\theta e_\theta$ are divergence-free, the compact feedback vanishes on this input, and $b_R=0$ on its support. Since $b\ge\varepsilon$, this gives
\[
\norm{S_{k,b_R}-S_{k,b}}\ge |k|\varepsilon,\qquad
\norm{M_{k,b_R}-M_{k,b}}\ge |k|\varepsilon .
\]
Thus the full seed operators have a nonvanishing norm difference even though their resolvents converge strongly. If the goal is only to find an unstable truncated profile, without tracking the original eigenvalue, the negative-direction criterion in the appendix gives a shorter existence argument.
\end{remark}

\subsection{Adding positive viscosity before bending}

Fix $R$ as provided by Theorem~\ref{thm:radial}. Let $D_k$ be the nonnegative self-adjoint realization of the negative vector Laplacian on $H_k$. If
$Ev=v_re_r+v_\theta e_\theta+v_ze_z$ denotes the associated Cartesian vector field on the cross-section, then
\[
D(D_k)=\{v\in H_k:Ev\in H^2(\R^2;\C^3)\},
\]
and
\begin{equation}\label{eq:Dform}
\ip{D_kv}{v}=\int_0^\infty\left(|v_r'|^2+|v_\theta'|^2+|v_z'|^2+
\frac{|v_r|^2+|v_\theta|^2}{r^2}+k^2|v|^2\right)r\dd r.
\end{equation}
The componentwise expressions for $-D_k$ are $\partial_r^2+r^{-1}\partial_r-r^{-2}-k^2$ for the first two components and $\partial_r^2+r^{-1}\partial_r-k^2$ for the third. We use the diagonal realization on the product space. At zero viscosity, the realization is the bounded operator
\[
M_{R,0}=M_{k,b_R}=S_{k,b_R}+K_{k,b_R},
\qquad D(M_{R,0})=\HH_k.
\]
With viscosity and resistivity both equal to $\nu>0$, the generator in physical time is
\[
M_{R,\nu}=-\nu D_k+S_{k,b_R}+K_{k,b_R},\qquad D(M_{R,\nu})=D(D_k)^2.
\]
\begin{theorem}\label{thm:viscous-fiber}
Every isolated eigenvalue of $M_{R,0}$ in the right half-plane is approximated, in each sufficiently small isolating disk, by eigenvalues of $M_{R,\nu}$ as $\nu\downarrow0$.
\end{theorem}
\begin{proof}
Set $A_\nu=-\nu D_k+S_{k,b_R}$. This operator is dissipative. At sufficiently large positive real parameters, surjectivity follows from $(z+\nu D_k)^{-1}$ and a Neumann series for the bounded perturbation. Thus $A_\nu$ is maximal dissipative and
\[
\norm{(z-A_\nu)^{-1}}\le(\Rea z)^{-1}.
\]
If $w\in D(D_k)^2$ and $y=(z-S_{k,b_R})w$, then
\[
(z-A_\nu)^{-1}y-w=-\nu(z-A_\nu)^{-1}D_kw\longrightarrow0.
\]
These values of $y$ form a dense set. Indeed, $z-S_{k,b_R}$ is a bounded isomorphism on $\HH_k$, and $D(D_k)^2$ is dense. This density argument avoids requiring the inviscid resolvent to preserve the viscous domain. Uniform boundedness of the resolvents therefore gives strong convergence for every input. The resolvent identity yields local uniformity in the spectral parameter. We conclude by applying Lemma~\ref{lem:transfer} with the fixed compact operator $K_{k,b_R}$.
\end{proof}

We now fix a sufficiently large $R$ and a sufficiently small but strictly positive $\nu$. These parameters remain fixed in the next section. The preceding argument is a singular-limit argument; it does not treat $\nu D_k$ as a small bounded perturbation on $L^2$. The field $b_R$ is still independent of the axial coordinate, and the corresponding periodic eigenmodes are still not whole-space $L^2$ functions. The finite-energy localization is carried out in the next section.

For this construction, linearization means the homogeneous derivative of the MHD evolution at a prescribed background. It does not require the radially truncated column to solve the unforced viscous equations. A fixed force may balance the background, and its derivative with respect to the perturbation is zero. The final force will be specified only after the similarity background has been fixed, so it is identical for the two solutions. In particular, the order of choices is not changed by passing from an ideal equilibrium to the forced viscous problem.
\section{From an infinite column to a finite-energy toroidal background}
\subsection{The full cross-sectional space and the local source equation}

Before bending the cylinder, we must restore all angular modes on the cross-section. Write $y=(y_1,y_2)$, let $s$ denote the axial coordinate, and set
\[
U_\infty(y)=\Omega(|y|)(-y_2,y_1,0),\qquad B_\infty(y)=b_R(|y|)(0,0,1).
\]
Define $\nabla_k=(\partial_{y_1},\partial_{y_2},ik)$, $\Delta_k=\Delta_y-k^2$, and
\[
H_k^{\rm all}=\{v\in L^2(\R^2;\C^3):\nabla_k\cdot v=0\},\qquad
\HH_k^{\rm all}=(H_k^{\rm all})^2.
\]
No axisymmetry condition is imposed on the cross-section. Let $P_k^{\rm all}$ be the Fourier multiplier
\begin{equation}\label{eq:Pkall}
\widehat{P_k^{\rm all}f}(\xi)=\left(I-\frac{(\xi_1,\xi_2,k)\otimes(\xi_1,\xi_2,k)}{|\xi|^2+k^2}\right)\widehat f(\xi),
\end{equation}
and let $\PP_k$ be its diagonal extension to the product space. Replacing axial derivatives in \eqref{eq:N} by $ik$, we obtain an expression denoted by $\N_{\infty,k}$. Consider
\begin{equation}\label{eq:Tflat}
T_\infty=\nu\Delta_k+\PP_k\N_{\infty,k},\qquad
D(T_\infty)=H^2(\R^2;\C^6)\cap\HH_k^{\rm all}.
\end{equation}
The coefficients of all first-order perturbations are compactly supported on the cross-section, and these perturbations have relative bound zero with respect to the Laplacian. Indeed, interpolation gives $\norm{\nabla v}_2\le\epsilon\norm{\Delta v}_2+C_\epsilon\norm v_2$, while the zeroth-order coefficients are bounded. Thus the operator is closed and densely defined. The operator $T_\infty$ acts on the cross-sectional space $L^2(\R^2)$ at the fixed axial frequency $k$; it is not itself an operator on $L^2(\R^3)$. We compare it with the toroidal operator through a common local source space, without identifying their global eigenspaces.

The following local source equation allows us to compare spectra on different spaces. Here $E_j$ inserts a source supported in a fixed bounded set $D$, $C_j$ is a local first-order observation, and $\PP_j$ is the corresponding divergence-free projection. The ambient spaces and the free-operator domains may depend on $j$; only the source space is fixed.
\begin{lemma}\label{lem:source}
For $j=0,1,\ldots$, let $X_j$ be a Hilbert space, $\HH_j\subset X_j$ a closed subspace, and $\PP_j:X_j\to\HH_j$ a bounded projection. Let $A_j$ be closed and densely defined on $\HH_j$, and suppose that a connected open set $O$ is contained in every $\rho(A_j)$. Equip $D(A_j)$ with the graph norm $\norm u_{D(A_j)}=\norm u_{\HH_j}+\norm{A_ju}_{\HH_j}$. On the fixed space $Y=L^2(D;\C^6)$, assume that
\[
E_j\in\mathcal B(Y,X_j),\qquad
C_j\in\mathcal B(D(A_j),Y).
\]
Write $B_j=\PP_jE_j$, $R_j(z)=(z-A_j)^{-1}$, and suppose that
\[
\mathcal K_j(z)=C_jR_j(z)B_j
\]
is compact for every $z\in O$, with $I-\mathcal K_j$ invertible at some point of $O$ for each $j$. Then $\mathcal K_j$ is analytic in operator norm, and the physical operator
\[
T_j=A_j+B_jC_j,\qquad D(T_j)=D(A_j),
\]
is closed and densely defined. For $z\in O$, the operator $I-\mathcal K_j(z)$ is invertible if and only if $z\in\rho(T_j)$; otherwise $z$ is an eigenvalue of $T_j$. If $\mathcal K_j\to\mathcal K_0$ locally uniformly in operator norm on $O$, every isolated point of noninvertibility of $I-\mathcal K_0$ is approximated by points of noninvertibility of $I-\mathcal K_j$ in each sufficiently small isolating disk.
\end{lemma}
\begin{proof}
The resolvent $R_j(z)$ is analytic from $\HH_j$ into the fixed graph space $D(A_j)$. Indeed, its usual resolvent identity together with $A_jR_j(z)=zR_j(z)-I$ gives analyticity in both components of the graph norm. Composing with the bounded maps $C_j$ and $B_j$ proves the asserted operator norm analyticity of $\mathcal K_j$. In particular, no bounded extension of the first-order map $C_j$ to $\HH_j$ is assumed.

Whenever $I-\mathcal K_j(z)$ is invertible, direct multiplication on $\HH_j$ and $D(A_j)$ gives the two-sided inverse
\[
 (z-T_j)^{-1}
 =R_j(z)+R_j(z)B_j
       (I-\mathcal K_j(z))^{-1}C_jR_j(z).
\]
Every term on the right is bounded from $\HH_j$ into $D(A_j)$ with its graph norm. At the point of invertibility assumed in the statement, this formula is an everywhere-defined bounded inverse on $\HH_j$. Its graph is closed, and exchanging the two graph coordinates proves that $T_j$ is closed. Dense definition follows from that of $A_j$.

Let $z_0$ be an isolated point of noninvertibility of $F_0(z)=I-\mathcal K_0(z)$, and choose an isolating circle $\Gamma$ centered at $z_0$, with its closed disk in $O$. Uniform norm convergence and a Neumann series imply that, for all sufficiently large $j$, $F_j$ is invertible on $\Gamma$ and
\[
\sup_{z\in\Gamma}\norm{F_j(z)^{-1}}\le M
\]
with $M$ independent of $j$. If $F_j$ were invertible everywhere inside $\Gamma$, its inverse would be analytic there. Cauchy's formula at the center would give $\norm{F_j(z_0)^{-1}}\le M$. Since $F_0(z_0)$ is Fredholm of index zero, choose a unit vector $f_0\in\ker F_0(z_0)$. But $F_j(z_0)f_0\to0$, so
\[
1=\norm{f_0}\le M\norm{F_j(z_0)f_0}\longrightarrow0,
\]
a contradiction. Thus every sufficiently large $j$ has a point of noninvertibility inside the circle.

At such a point, the Fredholm index is zero, so noninvertibility gives a nonzero kernel vector. Choose $0\ne f=\mathcal K_j(z)f$ and set
\[
u=(z-A_j)^{-1}\PP_jE_jf.
\]
If $u=0$, then $f=C_ju=0$, a contradiction. Moreover, the resolvent ensures that $u\in D(A_j)$, with
\[
 \norm u_{D(A_j)}
 \le\bigl[1+(1+|z|)\norm{R_j(z)}_{\HH_j\to\HH_j}\bigr]
      \norm{B_j}_{Y\to\HH_j}\norm f_Y.
\]
Since $f=C_ju$, we obtain $T_ju=zu$. Conversely, any eigenvector $u\in D(T_j)$ with eigenvalue $z\in O$ gives $f=C_ju\ne0$ and $f=\mathcal K_j(z)f$, because $C_ju=0$ would make $u$ a free eigenvector at a point of $\rho(A_j)$. This proves the spectral correspondence, since compactness makes $I-\mathcal K_j$ Fredholm of index zero at every point of $O$.

This argument uses one fixed source space only. In particular, it does not require insertion operators that preserve the global divergence constraint, or convergence of projections between the varying physical spaces. Nor does it require the graph-norm bounds for $C_j$ to be uniform in $j$; uniformity is imposed directly on the source families when proving their convergence.
\end{proof}

For the flat application, take $A_0=\nu\Delta_k$ on $\HH_k^{\rm all}$ and $C_0=\N_{\infty,k}$ followed by restriction to $D$. Its $H^2$ domain controls the local $H^1$ norm required by $C_0$. In the toroidal application, work on the physical Fourier sector, take $E_j=J_\ell E_\ell$, and observe fields by $C_\ell J_\ell^{-1}$ on $D$. The disk stays away from the cylindrical axis, so this observation is bounded on the graph domain of the sector Laplacian; no condition at the axis is added. In the small-dilation application of Section~4, Lemma~\ref{lem:domain} shows that the graph norm of $A_{\nu,\delta}$ controls $H^2$, and hence the local first-order map $\N_Q$. Thus all three uses of the lemma have the stated domain and observation properties. Each comparison uses its own fixed source space: a two-dimensional disk for the flat-to-toroidal limit, and a three-dimensional ball for the small-dilation limit. Their compactness and source-space convergence are verified separately below.

\begin{proposition}\label{prop:flat-all}
The unstable eigenvalue obtained in Theorem~\ref{thm:viscous-fiber} is also an isolated eigenvalue of finite algebraic multiplicity of \eqref{eq:Tflat} on the full cross-sectional space.
\end{proposition}
\begin{proof}
Expressing the radial eigenvector as a Cartesian field on the cross-section using the cylindrical basis gives an element of the domain in \eqref{eq:Tflat} satisfying the same equation. It remains to show that the spectral point is discrete in the enlarged space.

Choose a disk $D$ containing the coefficient supports at positive distance from its boundary, and let $E$ be extension by zero. For $\Rea z>0$, the free Stokes resolvent is
\[
R_\infty(z)=(z-\nu\Delta_k)^{-1}\PP_k.
\]
It maps $L^2$ into $H^2$. Consequently, $\mathcal K_\infty(z)=\N_{\infty,k}R_\infty(z)E$ is compact on $Y$: restrict first to a slightly larger disk, use the compact embedding $H^2\Subset H^1$, and then apply the first-order operator. The energy estimate gives
\begin{equation}\label{eq:source-large}
\norm{\mathcal K_\infty(z)}\le C\bigl((\Rea z)^{-1}+(\nu\Rea z)^{-1/2}\bigr).
\end{equation}
Thus $I-\mathcal K_\infty$ is invertible at sufficiently large positive real parameters. The analytic Fredholm theorem and the correspondence between eigenvectors prove discreteness and finite algebraic multiplicity. Equivalently, one can apply Fredholm theory directly to $I-\PP_k\N_{\infty,k}(z-\nu\Delta_k)^{-1}$ on $\HH_k^{\rm all}$: its perturbation is compact by the same local embedding. This also identifies the meromorphic inverse with the physical resolvent and gives finite-rank principal parts. Enlarging the space may increase the multiplicity of an eigenvalue, but does not affect this conclusion.
\end{proof}

\subsection{The toroidal background and exact Fourier sectors}

We now use cylindrical coordinates $(\rho,\vartheta,z)$ about the global axis. Set
\[
x=\rho-\ell,\qquad s=(x^2+z^2)^{1/2},\qquad \ell>2R,
\]
and define
\begin{equation}\label{eq:torus}
U_\ell=\frac\ell\rho\Omega(s)(-z e_\rho+x e_z),\qquad
B_\ell=\frac\rho\ell b_R(s)e_\vartheta.
\end{equation}
These fields are supported in the compact toroidal tube $s\le2R$, separated from the global axis.
\begin{lemma}\label{lem:geometry}
The pair $Q_\ell=(U_\ell,B_\ell)$ consists of real-valued, smooth, compactly supported divergence-free vector fields, and
\begin{equation}\label{eq:induction-zero}
U_\ell\cdot\nabla B_\ell-B_\ell\cdot\nabla U_\ell=0.
\end{equation}
Their energies are given by
\begin{align}\label{eq:torus-energy}
\norm{B_\ell}_2^2&=4\pi^2\ell\int_0^{2R}|b_R(s)|^2s\dd s
+\frac{6\pi^2}{\ell}\int_0^{2R}|b_R(s)|^2s^3\dd s,\notag\\
\norm{U_\ell}_2^2&=4\pi^2\ell^2\int_0^{2R}\frac{|\Omega(s)|^2s^3}{\sqrt{\ell^2-s^2}}\dd s<\infty.
\end{align}
\end{lemma}
\begin{proof}
The cross-sectional field $V=\Omega(s)(-z,x)$ is divergence-free and satisfies $V\cdot\nabla s=0$. Hence
\[
\rho^{-1}\partial_\rho(\rho U_{\ell,\rho})+\partial_zU_{\ell,z}
=\frac\ell\rho(\partial_xV_1+\partial_zV_2)=0.
\]
An axisymmetric, purely toroidal magnetic field is automatically divergence-free. The denominator $\rho$ does not vanish on the support, and extension by zero is smooth. Moreover,
\[
U_\ell\cdot\nabla B_{\ell,\vartheta}
=\ell^{-1}b_R(s)U_{\ell,\rho}
=\frac{B_{\ell,\vartheta}}\rho U_{\ell,\rho}.
\]
The right-hand side is precisely the toroidal component of $B_\ell\cdot\nabla U_\ell$, and all other components vanish. This proves \eqref{eq:induction-zero}.

On the cross-section, write $x=s\cos\alpha$ and $z=s\sin\alpha$. The volume element is $\rho s\dd s\dd\alpha\dd\vartheta$. Using
\[
\int_0^{2\pi}(\ell+s\cos\alpha)^3\dd\alpha=2\pi\ell^3+3\pi\ell s^2,
\qquad\int_0^{2\pi}\frac{\dd\alpha}{\ell+s\cos\alpha}=\frac{2\pi}{\sqrt{\ell^2-s^2}},
\]
direct integration gives \eqref{eq:torus-energy}.
\end{proof}

For every fixed finite $\ell$, the construction already gives a whole-space background of finite energy. Its energy generally grows linearly with $\ell$, but the proof requires only one sufficiently large, finite value of $\ell$. A fixed external force can balance the momentum equation for this background; the force in similarity variables will be specified in Section~5.

Now take $\ell=j/|k|$ and $n=j\operatorname{sgn}k$, so that $n/\ell=k$. The whole-space viscous operator is
\begin{equation}\label{eq:Tell}
T_\ell=\nu\Delta+\M_{Q_\ell},\qquad D(T_\ell)=H^2(\R^3;\C^6)\cap\HH.
\end{equation}
Since the background is axisymmetric about the global axis, the sector whose cylindrical components are proportional to $e^{in\vartheta}$ is invariant under $T_\ell$, $\Delta$, and $\PP$.

Let $H_\ell=\{(x,z):x>-\ell\}$, and order the components as $(\rho,z,\vartheta)$. The formula
\begin{equation}\label{eq:Jell}
(J_\ell v)(\rho,\vartheta,z)=\frac{e^{in\vartheta}}{\sqrt{2\pi\ell}}
\left(\frac\ell\rho\right)^{1/2}(v_1e_\rho+v_2e_z+v_3e_\vartheta)
\end{equation}
defines a unitary map from $L^2(H_\ell;\C^3)$ onto this vector sector before imposing the divergence constraint. We extend it diagonally to the product space. Direct calculation gives the exact transformed operator
\begin{equation}\label{eq:lapell}
\Delta_\ell=J_\ell^{-1}\Delta J_\ell
=\Delta_{x,z}+\rho^{-2}\begin{pmatrix}
-n^2-3/4&0&-2in\\0&-n^2+1/4&0\\2in&0&-n^2-3/4
\end{pmatrix},
\end{equation}
and the exact divergence constraint
\begin{equation}\label{eq:divell}
\partial_xv_1+\partial_zv_2+\frac{v_1}{2\rho}+\frac{in}\rho v_3=0.
\end{equation}
The domain is always the image under $J_\ell^{-1}$ of the whole-space $H^2$ sector; no artificial boundary condition is imposed at $x=-\ell$. On every fixed disk, $n/\rho\to k$ and $\rho^{-1}\to0$. Thus these differential coefficients converge to those of the flat problem at frequency $k$.

For completeness, the scalar weight in $J_\ell$ removes the first radial derivative because
\[
 \rho^{1/2}(\partial_\rho^2+\rho^{-1}\partial_\rho)
       (\rho^{-1/2}v)=\partial_x^2v+\frac{v}{4\rho^2}.
\]
The radial and azimuthal vector components each contribute a further $-\rho^{-2}$, and differentiating the cylindrical basis gives the off-diagonal entries in \eqref{eq:lapell}. These observations also explain why replacing the vector Laplacian by three scalar Laplacians before taking the limit would lose terms at finite $\ell$. All such terms are retained here.

\subsection{Resolvent convergence with pressure and spectral persistence}

Convergence of the differential coefficients alone does not control the nonlocal pressure. We therefore work directly with the whole-space Stokes resolvent and identify the limit using exactly divergence-free test fields.

Let $E_\ell$ extend sources in $Y$ by zero to $H_\ell$. For sufficiently large $\ell$, we have $D\Subset H_\ell$. Define
\begin{equation}\label{eq:Rell}
R_\ell(\lambda)=J_\ell^{-1}(\lambda-\nu\Delta)^{-1}\PP J_\ell E_\ell.
\end{equation}
\begin{lemma}\label{lem:local-stokes}
For every fixed disk $D_0\subset\R^2$ and compact set $\Lambda\Subset\{\Rea\lambda>0\}$,
\begin{equation}\label{eq:localnorm}
\sup_{\lambda\in\Lambda}\norm{
R_\ell(\lambda)|_{D_0}-(R_\infty(\lambda)E)|_{D_0}
}_{L^2(D)\to H^1(D_0)}\longrightarrow0.
\end{equation}
Here $\nu>0$ is fixed, and all projections are the exact projections on the corresponding whole spaces.
\end{lemma}
\begin{proof}
It suffices to consider a single vector field. Set $a=\inf_\Lambda\Rea\lambda>0$. The whole-space solution $u=(\lambda-\nu\Delta)^{-1}PF$ satisfies
\begin{equation}\label{eq:stokes-est}
\norm u_2\le a^{-1}\norm F_2,\quad
\norm{\nabla u}_2\le(\nu a)^{-1/2}\norm F_2,\quad
\norm{\Delta u}_2\le\nu^{-1}(1+|\lambda|/a)\norm F_2.
\end{equation}
The first two bounds are energy estimates; the third follows from $\nu\Delta u=\lambda u-PF$. Since $J_\ell$ is unitary, these estimates apply with $F=J_\ell E_\ell f$ and $\norm F_2=\norm f_2$.

We make the local bound explicit. Choose a fixed larger disk $D_1$ with $\overline{D_0}\subset D_1$, and write $\mathsf V_\ell$ for the matrix-valued potential in \eqref{eq:lapell}. For sufficiently large $\ell$, the quantities $|n|/\rho$ and $\rho^{-1}$ are uniformly bounded on $D_1$, so $\norm{\mathsf V_\ell}_{L^\infty(D_1)}\le C_{D_1,k}$. If $v_\ell=R_\ell(\lambda)f$, then
\[
 \Delta_{x,z}v_\ell=J_\ell^{-1}\Delta u-\mathsf V_\ell v_\ell
 \quad\hbox{on }D_1.
\]
Interior elliptic regularity therefore gives
\begin{equation}\label{s3-local-h2}
 \norm{v_\ell}_{H^2(D_0)}
 \le C_{D_0,D_1,k}\left[
 \frac1\nu\left(1+\frac{\sup_\Lambda|\lambda|}{a}\right)
 +\frac{1+C_{D_1,k}}a\right]\norm f_{L^2(D)}.
\end{equation}
The constant is independent of sufficiently large $\ell$ and of $\lambda\in\Lambda$; $\nu$ remains fixed. Taking one more intermediate disk if necessary, the compact embedding $H^2\Subset H^1$ gives collective compactness in $H^1(D_0)$. No boundary estimate near $x=-\ell$ is used.

We also allow the sources to vary weakly with $\ell$. Take $\ell_m\to\infty$, $\lambda_m\to\lambda\in\Lambda$, and $f_m\rightharpoonup f$ in $L^2(D)$. Extend $v_m$ by zero outside $H_{\ell_m}$, using this extension only as an $L^2$ function. The global $L^2$ bound and a diagonal application of local compactness yield a subsequence converging weakly on the whole plane to $v\in L^2(\R^2)$ and strongly in $H^1$ on each fixed disk. Equation~\eqref{eq:divell} implies $\nabla_k\cdot v=0$.

It is also necessary to retain a global form bound in the limit. On a fixed disk of radius $M$, differentiation of the weight gives
\begin{equation}\label{s3-planar-gradient}
 \norm{\nabla_{x,z}v_m}_{L^2(D_0)}
 \le\norm{\nabla u_m}_{L^2(\R^3)}
   +\frac{\norm{v_m}_{L^2(D_0)}}{2(\ell_m-M)}.
\end{equation}
Here the physical gradient includes the radial and vertical derivatives used on the left. The second term tends to zero. Passing to the limit on each disk, and then increasing the disk, proves that $v\in H^1(\R^2)$ with a bound depending only on $\nu,a$ and $\sup_m\norm{f_m}_2$. This reasoning avoids claiming that the zero extensions of $v_m$ have global $H^1$ regularity.

Let $w\in C_c^\infty(\R^2;\C^3)$ satisfy $\partial_xw_1+\partial_zw_2+ikw_3=0$. Choose the physical test field
\begin{equation}\label{eq:exacttest}
W_\ell=\frac{e^{in\vartheta}}{\sqrt{2\pi\ell}}
\left(\frac\ell\rho w_1e_\rho+\frac\ell\rho w_2e_z+w_3e_\vartheta\right).
\end{equation}
This field is exactly divergence-free: its divergence is the common prefactor multiplied by $(\ell/\rho)(\partial_xw_1+\partial_zw_2+ikw_3)$. Its support is separated from the axis, so it is an admissible smooth, compactly supported whole-space test field. Its transformed components are
\[
 z_\ell:=J_\ell^{-1}W_\ell
 =\left(\sqrt{\frac\ell\rho}w_1,
        \sqrt{\frac\ell\rho}w_2,
        \sqrt{\frac\rho\ell}w_3\right),
\]
which converge to $w$ in $H^1$ on a common compact support. Testing the physical Stokes equation against $W_\ell$ eliminates the pressure exactly. The diffusion pairing can be written, using \eqref{eq:lapell}, as
\[
 \ip{\nabla u_m}{\nabla W_{\ell_m}}
 =\ip{\nabla_{x,z}v_m}{\nabla_{x,z}z_{\ell_m}}
    -\ip{\mathsf V_{\ell_m}v_m}{z_{\ell_m}}.
\]
Since $\mathsf V_{\ell_m}\to-k^2I$ uniformly on this support, both terms have the asserted limit. The source pairing is $\ip{E f_m}{z_{\ell_m}}$, so weak convergence of the sources suffices. Passing to the limit gives
\begin{equation}\label{eq:weaklimit}
\lambda\ip vw+\nu\bigl(\ip{\nabla_yv}{\nabla_yw}+k^2\ip vw\bigr)=\ip{Ef}w.
\end{equation}
Thus the angular derivatives produce the $k^2$ term, while the curvature terms vanish without any separate convergence assertion for the pressure or the global projections.

These compactly supported divergence-free test fields are dense in $H^1\cap H_k^{\rm all}$. To see this, approximate an $H^1$ field by Schwartz fields and apply $P_k^{\rm all}$. Because $k\ne0$, its multiplier in \eqref{eq:Pkall} is smooth, preserves Schwartz regularity, and is bounded on $H^1$. It therefore suffices to approximate a divergence-free Schwartz field $w$. Multiplication by $\eta_L(y)=\eta(y/L)$ produces divergence $\nabla\eta_L\cdot(w_1,w_2)$. Adding
\[
-(ik)^{-1}\nabla\eta_L\cdot(w_1,w_2)
\]
to the third component cancels this divergence exactly. Its $H^1$ norm is bounded by a constant times $|k|^{-1}(L^{-1}+L^{-2})\norm w_{H^1}$ and tends to zero, while $\eta_Lw\to w$ in $H^1$. Thus \eqref{eq:weaklimit} holds on the entire form domain. Taking its real part with $w=v$ gives the coercive expression $(\Rea\lambda+\nu k^2)\norm v_2^2+\nu\norm{\nabla_yv}_2^2$. Uniqueness follows, and hence $v=R_\infty(\lambda)Ef$.

Finally, a contradiction argument upgrades this convergence to operator norm convergence. If \eqref{eq:localnorm} failed, there would be a fixed positive lower bound for the norm of the difference, witnessed by sources $f_m$ in the unit ball and parameters $\lambda_m\in\Lambda$. Extract subsequences with $f_m\rightharpoonup f$ and $\lambda_m\to\lambda$. The preceding argument shows that every locally convergent subsequence of the bent solutions has limit $R_\infty(\lambda)Ef$. Compactness therefore gives local strong $H^1$ convergence along a further subsequence. The flat solutions converge locally strongly to the same limit, by their $H^2$ bound and uniqueness. Their difference consequently tends to zero, contradicting the positive lower bound. This proves the norm convergence uniformly over the compact parameter set. The use of weakly varying sources is essential here; convergence for each fixed source alone would not imply the stated operator norm convergence.
\end{proof}

On the source space, set
\[
C_\ell=J_\ell^{-1}\N_{Q_\ell}J_\ell,
\]
and restrict its output to sources on $D$. Since the background vanishes near $\partial D$, no boundary distributions arise. All cross-sectional derivatives are controlled in $H^1$. By \eqref{eq:torus} and \eqref{eq:Jell}, the coefficients converge term by term in a neighborhood of $D$, with
\begin{equation}\label{eq:Cconv}
\norm{C_\ell-\N_{\infty,k}}_{H^1(D)\to L^2(D)}\longrightarrow0.
\end{equation}
For example, on the chosen mode we have the exact identity
\[
B_{\ell,\vartheta}\rho^{-1}\partial_\vartheta=ikb_R(s).
\]
The remaining derivatives of the basis vectors and the unitary weight contribute $O(\ell^{-1})$ terms. The backgrounds $U_\ell$ and $B_\ell$ converge uniformly to $\Omega(s)(-z,x,0)$ and $b_R(s)(0,0,1)$, respectively, together with the required first derivatives of their coefficients.

The following formulas make this coefficient comparison explicit. All components are now ordered as $(\rho,z,\vartheta)$, and $e_1,e_3$ denote the corresponding constant component vectors. Put
\[
 V=\Omega(s)(-z,x,0),\qquad
 A=\frac\ell\rho V,\qquad \beta=\frac\rho\ell b_R(s),
 \qquad \mathsf Gv=(-v_3,0,v_1).
\]
For a three-component amplitude $v$, define the transformed transport operators $\mathcal T_U^\ell v=J_\ell^{-1}(U_\ell\cdot\nabla J_\ell v)$ and $\mathcal T_B^\ell v=J_\ell^{-1}(B_\ell\cdot\nabla J_\ell v)$. Direct differentiation gives
\begin{align}\label{s3-transport-formulas}
 \mathcal T_U^\ell v
 &= A_1\partial_xv+A_2\partial_zv-\frac{A_1}{2\rho}v,\notag\\
 \mathcal T_B^\ell v
 &=ikb_R(s)v+\frac{b_R(s)}\ell\mathsf Gv.
\end{align}
Likewise, the two stretching operators are
\begin{align}\label{s3-stretch-formulas}
 \mathcal S_U^\ell v
 &:=J_\ell^{-1}\bigl((J_\ell v)\cdot\nabla U_\ell\bigr)
   =(v_1\partial_x+v_2\partial_z)A+\frac{A_1}{\rho}v_3e_3,\notag\\
 \mathcal S_B^\ell v
 &:=J_\ell^{-1}\bigl((J_\ell v)\cdot\nabla B_\ell\bigr)
   =(v_1\partial_x\beta+v_2\partial_z\beta)e_3
     -\frac{\beta}{\rho}v_3e_1.
\end{align}
The matrix $\mathsf G$ accounts for differentiating $e_\rho$ and $e_\vartheta$ in the angular direction. In the stretching expressions the Fourier factor belongs to the perturbation, whereas the differentiated background has no angular dependence in its components. These formulas therefore include different angular contributions, as required by the linearization.

In this notation the two components of $C_\ell(v,h)$ are
\[
 -\mathcal T_U^\ell v-\mathcal S_U^\ell v
       +\mathcal T_B^\ell h+\mathcal S_B^\ell h,
 \qquad
 -\mathcal T_U^\ell h-\mathcal S_B^\ell v
       +\mathcal T_B^\ell v+\mathcal S_U^\ell h.
\]
On the fixed disk $D$, $A\to V$ and $\beta\to b_R$ in $C^1$, with error $O(\ell^{-1})$, and every additional displayed curvature or weight term is $O(\ell^{-1})$. Multiplication by a bounded coefficient and a first derivative maps $H^1(D)$ to $L^2(D)$. This proves \eqref{eq:Cconv}, in fact with rate $O(\ell^{-1})$ for these fixed profiles and this fixed disk.

We can now complete the spectral transfer to a finite-energy background.
\begin{theorem}\label{thm:bending}
Let $\lambda_\infty$ be the unstable eigenvalue in Proposition~\ref{prop:flat-all}, and let $\Gamma$ be a sufficiently small isolating circle around it in the right half-plane. For all sufficiently large $j$, \eqref{eq:Tell} has an eigenvalue inside $\Gamma$, which is isolated and has finite algebraic multiplicity on the whole space $\HH$.
\end{theorem}
\begin{proof}
Let $\widehat T_\infty=T_\infty$, and let $\Pi_\infty$ be its Riesz projection associated with the counterclockwise circle $\Gamma$. Choose a normalized eigenvector $u_\infty$ for $\lambda_\infty$. Since $\Pi_\infty$ is independent of the source disk, we may first choose a disk $D$ containing all cross-sectional coefficient supports with positive margin and satisfying
\begin{equation}\label{s3-source-tail}
 \eta_D:=\norm{u_\infty}_{L^2(\R^2\setminus D)},
 \qquad (1+\norm{\Pi_\infty})\eta_D<\frac12.
\end{equation}
Keep this disk fixed, put $Y=L^2(D;\C^6)$, and take $\ell$ sufficiently large that $D\Subset H_\ell$.

Let $\HH_n\subset\HH$ denote the solenoidal physical $n$th angular sector, and put $\mathcal X_\ell=J_\ell^{-1}\HH_n$ and $\mathcal X_\infty=\HH_k^{\mathrm{all}}$. On these respective spaces, let $A_\ell$ be the pullback of $\nu\Delta|_{\HH_n}$ and $A_\infty=\nu\Delta_k$, with domains
\[
 D(A_\ell)=J_\ell^{-1}\bigl(H^2(\R^3;\C^6)\cap\HH_n\bigr),\qquad
 D(A_\infty)=H^2(\R^2;\C^6)\cap\mathcal X_\infty.
\]
In particular, no boundary condition is added at $x=-\ell$. Set
\[
 \PP_\ell=J_\ell^{-1}\PP J_\ell,\qquad
 \PP_\infty=\PP_k,\qquad E_\infty=E,\qquad
 S_q=\PP_qE_q\quad(q=\ell,\infty).
\]
The observations $C_\ell$ and $C_\infty=\N_{\infty,k}|_D$ take values in $Y$. The support choice gives the exact identities
\[
 \widehat T_q=A_q+S_qC_q,\qquad
 D(\widehat T_q)=D(A_q),\qquad
 \widehat T_\ell=J_\ell^{-1}(T_\ell|_{\HH_n})J_\ell.
\]
For $\Rea\lambda>0$, define
\[
 G_q(\lambda)=(\lambda-A_q)^{-1},\qquad
 \mathcal R_q(\lambda)=G_q(\lambda)S_q,\qquad
 F_q(\lambda)=I-C_q\mathcal R_q(\lambda).
\]
Thus $\mathcal R_\ell=R_\ell$ and $\mathcal R_\infty=R_\infty E$. Since $C_q$ is bounded for the graph norm of $A_q$, these source families are analytic in operator norm. Local $H^2$ estimates and the compact embedding $H^2\Subset H^1$ show that $I-F_q$ is compact on $Y$.

Lemma~\ref{lem:local-stokes} and \eqref{eq:Cconv} imply
\[
 \sup_{\lambda\in\Gamma}
 \norm{F_\ell(\lambda)-F_\infty(\lambda)}_{Y\to Y}\longrightarrow0.
\]
For $\lambda\in\Gamma$, a nonzero vector in $\ker F_\infty(\lambda)$ would reconstruct a nonzero eigenvector $\mathcal R_\infty(\lambda)f$ of $T_\infty$, contrary to the choice of $\Gamma$. Since $F_\infty$ is identity minus compact, it is therefore invertible on $\Gamma$. Compactness of $\Gamma$ and the preceding norm convergence give, for all sufficiently large $\ell$,
\begin{equation}\label{s3-source-inverses}
 \sup_{\lambda\in\Gamma}\norm{F_\ell(\lambda)^{-1}}<\infty,
 \qquad
 \sup_{\lambda\in\Gamma}
 \norm{F_\ell(\lambda)^{-1}-F_\infty(\lambda)^{-1}}
 \longrightarrow0,
\end{equation}
where the first bound is uniform in $\ell$.

Solving $u=G_qg+\mathcal R_q C_qu$ gives the full resolvent identities
\begin{align}\label{s3-sector-resolvent}
 (\lambda-\widehat T_q)^{-1}
 &=G_q+\mathcal R_qF_q^{-1}C_qG_q,\notag\\
 (\lambda-\widehat T_q)^{-1}S_q
 &=\mathcal R_qF_q^{-1}.
\end{align}
These formulas also verify membership in $D(A_q)$ and invertibility on both sides. They give the required bound for the complete $n$-sector resolvent. Indeed, with $a=\min_{\Gamma}\Rea\lambda>0$, the global energy estimates \eqref{eq:stokes-est} yield
\[
 \norm{G_\ell},\ \norm{\mathcal R_\ell}\le a^{-1},\qquad
 \norm{C_\ell G_\ell}_{\mathcal X_\ell\to Y}
 \le C\bigl(a^{-1}+(\nu a)^{-1/2}\bigr).
\]
The last estimate follows from the physical gradient estimate and the uniformly bounded first-order background coefficients. Consequently,
\begin{equation}\label{s3-sector-uniform}
 \sup_{\ell\ge\ell_0}\sup_{\lambda\in\Gamma}
 \norm{(\lambda-\widehat T_\ell)^{-1}}_{\mathcal X_\ell\to\mathcal X_\ell}
 <\infty.
\end{equation}

Define the genuine Riesz projections
\[
 \Pi_q=\frac1{2\pi i}\int_\Gamma
       (\lambda-\widehat T_q)^{-1}\,\dd\lambda.
\]
Integrating the second identity in \eqref{s3-sector-resolvent}, and using \eqref{s3-source-inverses} and Lemma~\ref{lem:local-stokes}, gives, for every fixed disk $D_0$,
\begin{equation}\label{s3-local-projection-convergence}
 (\Pi_\ell S_\ell)|_{D_0}
 \longrightarrow(\Pi_\infty S_\infty)|_{D_0}
 \quad\text{in }\mathcal L\bigl(Y,H^1(D_0;\C^6)\bigr).
\end{equation}
Thus the contour integrals defining the projections are those of the full resolvents; the localized operators in \eqref{s3-local-projection-convergence} are used only for comparison on a common space.

For $f=u_\infty|_D$, the identities $\PP_k u_\infty=u_\infty$ and $\Pi_\infty u_\infty=u_\infty$ imply
\[
 \norm{(\Pi_\infty S_\infty f)|_D}_{L^2(D)}
 \ge \norm{u_\infty|_D}_2-\norm{\Pi_\infty}\eta_D
 \ge 1-(1+\norm{\Pi_\infty})\eta_D>\frac12.
\]
Taking $D_0=D$ in \eqref{s3-local-projection-convergence} proves $\Pi_\ell\ne0$ for all sufficiently large $\ell$, so the $n$-sector has spectrum inside $\Gamma$.

Finally, for each fixed finite $\ell$, the operator
\[
 \M_{Q_\ell}(\lambda-\nu\Delta)^{-1}:\HH\longrightarrow\HH
\]
is compact: its unprojected output has compact support, local $H^2\Subset H^1$ gives compactness for the first-order terms, and the bounded Leray projection preserves compactness. Its norm tends to zero as $\lambda\to+\infty$. Analytic Fredholm theory and
\[
 \lambda-T_\ell=
 [I-\M_{Q_\ell}(\lambda-\nu\Delta)^{-1}]
 (\lambda-\nu\Delta)
\]
therefore show that the whole-space spectrum in $\Rea\lambda>0$ consists of isolated eigenvalues with finite-rank Riesz projections. The $n$-sector is a reducing subspace of $T_\ell$. Hence the spectral point detected above is an eigenvalue of the whole-space operator with finite algebraic multiplicity, and its eigenvector is recovered by $J_\ell$ as an element of $H^2(\R^3;\C^6)\cap\HH$.
\end{proof}

Choosing one sufficiently large but finite $\ell$, Lemma~\ref{lem:geometry} and Theorem~\ref{thm:bending} prove Theorem~\ref{thm:physical}. The viscosity remains fixed and positive throughout this section; no estimates uniform as $\nu\downarrow0$ are required during bending. Closing the cylinder into a torus puts both the background and the eigenvector in whole-space $L^2$, without treating the infinite cylinder as a finite-energy whole-space background.
\section{An unstable background in similarity variables}
\subsection{The free generator and the small-dilation limit}

Fix the background $Q=(U,B)$ and the viscosity $\nu>0$ chosen in the preceding section. The operator $T_0=\nu\Delta+\M_Q$ has an unstable discrete eigenvalue. We introduce a further parameter $\delta\ge0$ and set
\begin{equation}\label{eq:Adelta}
A_{\nu,\delta}=\nu\Delta+\frac\delta2(1+y\cdot\nabla),\qquad
T_\delta=A_{\nu,\delta}+\M_Q.
\end{equation}
The free generator is defined by the strongly continuous semigroup
\begin{equation}\label{eq:free-semigroup}
(e^{tA_{\nu,\delta}}w)(y)
=e^{\delta t/2}\left(e^{\nu(e^{\delta t}-1)\Delta/\delta}w\right)(e^{\delta t/2}y),
\qquad \norm{e^{tA_{\nu,\delta}}}_{2\to2}\le e^{-\delta t/4},
\end{equation}
when $\delta>0$; for $\delta=0$, we use the heat semigroup. The formula follows by differentiating the heat time and the spatial dilation. Its norm estimate uses the factor $e^{-3\delta t/4}$ from the change of spatial variables. Heat evolution and dilation preserve the divergence constraint, so the formula defines a semigroup on $\HH$.
Strong continuity follows first for Schwartz data from the heat and dilation formulas, and then for all $L^2$ data by density and the uniform norm bound near $t=0$. In particular, this definition fixes the realization of the free operator before any perturbation is introduced. We shall not choose a different realization when taking the singular limit $\delta\downarrow0$.

\begin{lemma}\label{lem:domain}
On $\HH$,
\begin{align}\label{eq:domains}
D(A_{\nu,0})&=H^2\cap\HH,\notag\\
D(A_{\nu,\delta})&=\{w\in H^2\cap\HH:y\cdot\nabla w\in L^2\},\qquad\delta>0.
\end{align}
The operator $\M_Q$ has relative bound zero with respect to each $A_{\nu,\delta}$. Consequently, $T_\delta$ is closed and densely defined on the same domain.
\end{lemma}
\begin{proof}
For smooth, rapidly decaying fields in three dimensions, integration by parts gives
\[
\Rea\ip{\Delta w}{y\cdot\nabla w}=\tfrac12\norm{\nabla w}_2^2.
\]
Thus, for $c>0$,
\begin{equation}\label{eq:graph}
\norm{\nu\Delta w+c\,y\cdot\nabla w}_2^2
=\nu^2\norm{\Delta w}_2^2+c^2\norm{y\cdot\nabla w}_2^2+\nu c\norm{\nabla w}_2^2.
\end{equation}
After adding the $L^2$ norm, this identity proves equivalence of the graph norm and the intersection norm in \eqref{eq:domains}. The scalar zeroth-order term $\delta/2$ does not affect this equivalence.

We give the domain identification explicitly. Let $\mathcal C=\mathcal S(\R^3;\C^6)\cap\HH$, where $\mathcal S$ denotes the Schwartz space. This subspace is dense in $\HH$, invariant under \eqref{eq:free-semigroup}, and contained in the generator domain. A dense invariant subspace of the domain of a strongly continuous semigroup is a core. In the present setting, this fact can also be checked directly: if a continuous linear functional annihilates $(z-A_{\nu,\delta})\mathcal C$, then, for $w\in\mathcal C$, its value on $e^{tA_{\nu,\delta}}w$ satisfies $h'(t)=zh(t)$. For positive real $z$, the semigroup bound forces $h(0)=0$. Density of $\mathcal C$ shows that this functional is zero, so $(z-A_{\nu,\delta})\mathcal C$ is dense. Let $B$ be the closure of the restriction to $\mathcal C$. Then $B\subset A_{\nu,\delta}$, and the estimate $\|(z-B)w\|_2\ge z\|w\|_2$ and closedness of $B$ make $\Ran(z-B)$ closed. This range contains the dense set just identified and is therefore all of $\HH$. Since $z$ also belongs to the resolvent set of $A_{\nu,\delta}$, the two inverses agree. Hence $B=A_{\nu,\delta}$, proving the core assertion.

The space $\mathcal C$ is also dense in the claimed intersection norm. For $w$ in that intersection, multiplication by a spatial cutoff $\chi(y/L)$ converges to $w$ in $H^2$ and in the norm of $y\cdot\nabla$. In the latter norm, the additional term is $(y\cdot\nabla\chi(y/L))w$, controlled by the $L^2$ tail of $w$. Next use smooth Fourier cutoffs away from zero and infinity. Their commutators with the dilation generator have multipliers $-\xi\cdot\nabla\chi$, uniformly bounded and supported in shrinking or escaping frequency annuli; the errors tend to zero in $L^2$. Applied after the spatial cutoff, these operations produce Schwartz fields. Finally apply $\PP$: this multiplier is smooth away from zero, preserves the resulting Schwartz class, and commutes with the Laplacian and dilation because it is homogeneous of degree zero. Thus the approximations lie in $\mathcal C$. The graph identity now identifies the closure on this core with precisely \eqref{eq:domains}.

More explicitly, the approximating fields can be chosen as
\[
\PP\,\theta_{\varepsilon,M}(D)\bigl(\chi(y/L)w\bigr),
\]
where $\theta_{\varepsilon,M}$ has compact Fourier support in $\{\varepsilon<|\xi|<2M\}$. Take the frequency limits for each fixed $L$, and then let $L\to\infty$. Although spatial truncation temporarily destroys the divergence constraint, applying $\PP$ restores it without changing the limit in either graph component. Compact spatial support before Fourier truncation ensures smoothness of the Fourier transform, so the final band-limited fields really are Schwartz. Conversely, convergence in the intersection norm implies convergence of the differential expression, and distributional closedness of $\Delta$ and $y\cdot\nabla$ rules out a larger closure. These observations justify both inclusions in the domain equality.

For $\alpha=\Rea z+\delta/4>0$, integration by parts and closure give
\begin{equation}\label{eq:delta-energy}
\alpha\norm w_2^2+\nu\norm{\nabla w}_2^2
=\Rea\ip{(z-A_{\nu,\delta})w}w.
\end{equation}
The coefficient $\delta/4$ comes from
$\Rea\ip{y\cdot\nabla w}w=-3\norm w_2^2/2$ together with the scalar term $\delta/2$. Thus, for $w=(z-A_{\nu,\delta})^{-1}f$,
\begin{equation}\label{s4-free-resolvent-bounds}
\norm w_2\le\alpha^{-1}\norm f_2,
\qquad\norm{\nabla w}_2\le(\nu\alpha)^{-1/2}\norm f_2.
\end{equation}
Existence of this resolvent throughout $\Rea z>-\delta/4$ also follows from the Laplace transform of \eqref{eq:free-semigroup}; hence the energy calculation is an estimate for that fixed generator, not an independent choice of extension.
Since $\norm{\M_Qw}_2\le C_Q(\norm{\nabla w}_2+\norm w_2)$ and
\[
\nu\norm{\nabla w}_2^2
\le\norm{A_{\nu,\delta}w}_2\norm w_2,
\]
Young's inequality yields, for every $\epsilon>0$,
\[
\norm{\M_Qw}_2\le\epsilon\norm{A_{\nu,\delta}w}_2
+C_{\epsilon,Q,\nu,\delta}\norm w_2.
\]
Relatively bounded perturbation theory \cite{Kato} therefore gives closedness on the stated domain.
\end{proof}

The unbounded coefficient $\delta y$ prevents treating dilation as a small perturbation in the graph norm of $T_0$. Instead, we use local source operators and retain global energy control throughout the limit.
\begin{lemma}\label{lem:dilation}
Every isolated eigenvalue of $T_0$ in the right half-plane is approximated, as $\delta\downarrow0$, by eigenvalues of $T_\delta$ lying in any sufficiently small isolating disk.
\end{lemma}
\begin{proof}
Choose a ball $D\Subset\R^3$ containing the supports of $Q$ and its derivatives at positive distance from its boundary. Let $Y=L^2(D;\C^6)$ and let $E$ denote extension by zero. For $\Rea\lambda>0$, set
\begin{equation}\label{s4-source-dilation}
u_\delta=(\lambda-A_{\nu,\delta})^{-1}\PP Ef,
\qquad\mathcal K_\delta(\lambda)f=\N_Qu_\delta.
\end{equation}
The unprojected expression on the right is supported in $D$ and is regarded as an element of $Y$.

If $\lambda$ belongs to a compact subset $\Lambda$ of the right half-plane, \eqref{eq:delta-energy} gives a uniform global $H^1$ bound for $u_\delta$ in terms of $\norm f_Y$. For $0\le\delta\le1$, the equation
\[
\nu\Delta u_\delta=(\lambda-\delta/2)u_\delta
-(\delta/2)y\cdot\nabla u_\delta-\PP Ef
\]
then gives uniform local $H^2$ bounds by interior elliptic estimates. On each fixed ball the coefficient $\delta y$ is bounded, while $\PP Ef$ has a global $L^2$ bound. No assertion that the projection preserves compact support is used.
In quantitative form, for concentric balls $D\Subset D_1\Subset D_2$, there is a constant depending on $D_1,D_2,\Lambda$ and the fixed $\nu$, but independent of $\delta\in[0,1]$, such that
\begin{equation}\label{s4-local-h2}
\norm{u_\delta}_{H^2(D_1)}
\le C\left(\norm{u_\delta}_{H^1(D_2)}+\norm f_Y\right)
\le C'\norm f_Y.
\end{equation}
The larger balls are used only to apply an interior estimate. They impose no boundary condition on the resolvent solution. In particular, the constants are allowed to depend on the ball, whereas the weak global $H^1$ compactness comes from \eqref{s4-free-resolvent-bounds}. This distinction prevents replacing a global uniqueness argument by an insufficient local one.

We check convergence with weakly varying sources, as required for operator norm convergence. Suppose that $f_j\rightharpoonup f$ in $Y$, $\delta_j\downarrow0$, and $\lambda_j\to\lambda\in\Lambda$. Passing to a subsequence, the global bound gives weak convergence in $H^1(\R^3)$, and the local $H^2$ bounds give strong convergence in $H^1$ on each fixed ball. Against a compactly supported test field, the dilation term tends to zero, whereas $\PP Ef_j\rightharpoonup\PP Ef$ globally in $L^2$. The limit therefore satisfies
\[
(\lambda-\nu\Delta)u=\PP Ef.
\]
It remains divergence-free by weak convergence. The global $L^2$ solution is unique, either by Fourier transformation or by the energy estimate, and equals $(\lambda-\nu\Delta)^{-1}\PP Ef$. Hence every convergent subsequence has this same limit.

For completeness, if local operator norm convergence failed, there would be unit sources $f_j$, parameters $\lambda_j\in\Lambda$, and $\delta_j\downarrow0$ for which the local $H^1$ difference stayed bounded away from zero. Extract weakly convergent sources and convergent parameters. The preceding argument gives strong local convergence of the dilated solutions. The flat solutions have the same strong local limit, by their local $H^2$ bounds and uniqueness. This contradiction proves
\begin{equation}\label{s4-resolvent-local}
\sup_{\lambda\in\Lambda}\norm{
\left[(\lambda-A_{\nu,\delta})^{-1}-(\lambda-\nu\Delta)^{-1}\right]\PP E
}_{Y\to H^1(D)}\longrightarrow0.
\end{equation}
Since $\N_Q:H^1(D)\to Y$ is bounded, $\mathcal K_\delta\to\mathcal K_0$ locally uniformly in operator norm. Local $H^2$ regularity and Rellich's theorem show that each operator is compact. For each fixed $\delta$, Lemma~\ref{lem:domain} makes $\N_Q:D(A_{\nu,\delta})\to Y$ graph-norm bounded; composing it with the graph-norm analytic free resolvent proves operator norm analyticity, as in Lemma~\ref{lem:source}. These two properties therefore follow from distinct arguments. Moreover, \eqref{eq:delta-energy} gives the uniform estimate
\[
\norm{\mathcal K_\delta(\lambda)}
\le C_Q\bigl((\Rea\lambda)^{-1}+(\nu\Rea\lambda)^{-1/2}\bigr).
\]
Thus $I-\mathcal K_\delta$ is simultaneously invertible at one sufficiently large positive real parameter. Lemma~\ref{lem:source} applies. Its eigenvectors are reconstructed using the free resolvent and hence belong to the actual domain $D(T_\delta)$, including the dilation condition when $\delta>0$.
To spell out the correspondence, a nonzero solution $f=\mathcal K_\delta(\lambda)f$ gives $u=(\lambda-A_{\nu,\delta})^{-1}\PP Ef$, with $f=\N_Qu$. The vector $u$ cannot vanish, and $(\lambda-T_\delta)u=0$. Conversely, an eigenvector in this half-plane gives a nonzero local source, since a zero source would make it a free eigenvector in the free resolvent set. The source reduction therefore detects genuine eigenvalues rather than approximate modes. The argument establishes persistence in each prescribed isolating disk; it does not require convergence in operator norm of the free resolvents on the entire whole-space $L^2$ space.
\end{proof}

\subsection{Exact rescaling and discreteness of the spectrum}

Fix a value $\delta>0$ furnished by Lemma~\ref{lem:dilation}, and set
\begin{equation}\label{eq:scaling}
a_0=\sqrt{\nu/\delta},\qquad
\overline Q(x)=\frac1{\sqrt{\nu\delta}}Q(a_0x),\qquad
(D_{a_0}w)(x)=a_0^{3/2}w(a_0x).
\end{equation}
The operator $D_{a_0}$ is unitary on $\HH$. Homogeneity of the Laplacian, the first-order expression $\M_Q$, and the Leray projection gives the exact identity
\begin{equation}\label{eq:exactscaling}
D_{a_0}^{-1}\left[\Delta+\tfrac12(1+x\cdot\nabla)+\M_{\overline Q}\right]D_{a_0}
=\delta^{-1}\left[\nu\Delta+\tfrac\delta2(1+y\cdot\nabla)+\M_Q\right].
\end{equation}
Indeed, the Laplacian acquires the factor $a_0^2=\nu/\delta$, and both first-order background terms acquire $a_0/\sqrt{\nu\delta}=1/\delta$. The identity also holds at the level of domains: dilation is an isomorphism of $H^2$, and
$x\cdot\nabla D_{a_0}w=D_{a_0}(y\cdot\nabla w)$. Thus the intersection domains in Lemma~\ref{lem:domain} correspond exactly. An unstable eigenvalue $\lambda_\delta$ of $T_\delta$ becomes $\lambda_\delta/\delta$ for $L$.

The background $\overline Q$ remains real-valued, smooth, compactly supported, and divergence-free, with
\begin{equation}\label{eq:Qbarenergy}
\norm{\overline Q}_2^2=\delta^{1/2}\nu^{-5/2}\norm Q_2^2<\infty.
\end{equation}
No uniform bound in the previously fixed parameters is required here; only one final background is used.
For example, differentiation gives $\diver\overline U=(a_0/\sqrt{\nu\delta})(\diver U)(a_0x)=0$, and the same holds for $\overline B$. The $L^2$ identity follows from the squared amplitude $(\nu\delta)^{-1}$ and the volume factor $a_0^{-3}$. Thus the amplitude enlargement needed to normalize viscosity is compatible with finite energy and preserves the nonzero magnetic field.

\begin{proposition}\label{prop:ss-discrete}
Let $A=\Delta+\frac12(1+x\cdot\nabla)$ and $L=A+\M_{\overline Q}$, with domain as in \eqref{eq:Lss}. The spectrum of $L$ in $\{\Rea z>-1/4\}$ consists entirely of isolated eigenvalues of finite algebraic multiplicity, with no accumulation point in this open half-plane.
\end{proposition}
\begin{proof}
The free semigroup satisfies $\norm{e^{tA}}\le e^{-t/4}$. If $u=(z-A)^{-1}f$ and $\alpha=\Rea z+1/4>0$, then
\[
\alpha\norm u_2^2+\norm{\nabla u}_2^2=\Rea\ip fu,
\qquad\norm u_2\le\alpha^{-1}\norm f_2,
\qquad\norm{\nabla u}_2\le\alpha^{-1/2}\norm f_2.
\]
On each fixed ball, $\Delta u=(z-1/2)u-(1/2)x\cdot\nabla u-f$ gives a local $H^2$ bound. First apply Rellich's theorem to the unprojected first-order expression with compactly supported coefficients, and then apply the globally bounded $\PP$. This proves compactness of $\M_{\overline Q}(z-A)^{-1}$. Furthermore,
\[
\norm{\M_{\overline Q}(z-A)^{-1}}
\le C(\alpha^{-1}+\alpha^{-1/2})\longrightarrow0
\quad(\Rea z\to\infty).
\]
The factorization
\[
z-L=\left[I-\M_{\overline Q}(z-A)^{-1}\right](z-A)
\]
on $D(A)$ and the analytic Fredholm theorem prove the claim.
\end{proof}

The parameter order is fixed: choose $R$ large, then $\nu$ small and positive, next $\ell$ large and finite, and finally $\delta$ small and positive. Every preceding parameter is kept fixed when the next limit is taken. The exact transformation \eqref{eq:exactscaling} completes the construction without an additional assumption of ideal toroidal instability.

\subsection{Semigroup growth and smoothing}

To complete Theorem~\ref{thm:ss}, we establish the growth and regularity bounds needed in the nonlinear argument. The essential growth bound, rather than only the essential spectrum of the generator, is the relevant quantity.
\begin{proposition}\label{prop:semigroup}
The operator $L$ generates a strongly continuous semigroup on $\HH$, and $\omega_{\ess}(L)\le-1/4$. Its spectral bound $a=s(L)>0$ is finite and is attained by an eigenvalue $\lambda=a+i\omega$ with a nonzero eigenvector $\eta\in\bigcap_{m\ge0}H^m$. Moreover, \eqref{eq:main-smoothing} holds.
\end{proposition}
\begin{proof}
The free formula implies
\[
\norm{\M_{\overline Q}e^{tA}}_{2\to2}\le C(1+t^{-1/2}),\qquad0<t\le1.
\]
Its integral on $(0,t_0)$ tends to zero with $t_0$. The integrable semigroup perturbation theorem, equivalently the convergent Volterra iteration, therefore yields a strongly continuous semigroup generated by $A+\M_{\overline Q}$ on $D(A)$. Lemma~\ref{lem:domain} has already established closedness on this domain.
Indeed, the perturbation criterion is satisfied with a constant bounded by $C(t_0+2\sqrt{t_0})<1$ for small $t_0$. There is no requirement that $\M_{\overline Q}$ be bounded on $L^2$. The resulting generator has the prescribed domain: for large positive $z$, its Laplace resolvent agrees with the resolvent obtained from the convergent Neumann series for $I-\M_{\overline Q}(z-A)^{-1}$. The latter factorization acts on $D(A)$, so it identifies the generated operator with the closed realization already defined above.

For every $s>0$, $\M_{\overline Q}e^{sA}$ is compact. Before projection, its kernels are products of compactly supported smooth coefficients with scaled Gaussian kernels or their first derivatives. They are square-integrable in the input and output variables and hence define Hilbert--Schmidt operators. Applying $\PP$ preserves compactness. Directly from these kernels, the family is continuous in operator norm on compact intervals of positive times.
For instance, the scalar free kernel is $e^{s/2}G_{e^s-1}(e^{s/2}x-y)$, where $G_t$ is the heat kernel. Integration in $y$ of its square or the square of one $x$ derivative is finite for every $s>0$; the coefficient restricts the remaining $x$ integration to a compact set. This verifies compactness for the transport terms as well as for the zeroth-order stretching and magnetic-coupling terms. No compactness of the free semigroup itself is asserted.

The Duhamel identity reads
\begin{equation}\label{eq:duhamel-compact}
e^{tL}-e^{tA}=\int_0^t e^{(t-s)L}\M_{\overline Q}e^{sA}\dd s.
\end{equation}
Removing $(0,\varepsilon)$ changes the integral in operator norm by at most $C_t(\varepsilon+\sqrt\varepsilon)$. On $[\varepsilon,t]$, its integrand is a norm-continuous compact-operator-valued function: a strongly continuous, locally bounded semigroup acts norm-continuously on the relatively compact image of a compact operator. The truncated integral is therefore an operator norm limit of compact Riemann sums. It follows that the full difference in \eqref{eq:duhamel-compact} is compact. Consequently,
\[
\norm{e^{tL}}_{\ess}=\norm{e^{tA}}_{\ess}\le e^{-t/4},
\qquad\omega_{\ess}(L)\le-1/4.
\]

Engel--Nagel \cite[Chapter IV, Corollary 2.11]{EN}, also used in \cite[Lemma 4.2]{ABC}, now gives
\begin{equation}\label{eq:growthbound}
\omega_0(L)=\max\{\omega_{\ess}(L),s(L)\}.
\end{equation}
The same result states that, for every $w>\omega_{\ess}(L)$, the spectrum in $\{\Rea z\ge w\}$ is finite and its spectral projection has finite rank. A strongly continuous semigroup has a finite exponential growth bound; the previously constructed positive eigenvalue therefore implies that $a=s(L)>0$ is finite and attained. In particular, for every $\epsilon>0$,
\begin{equation}\label{eq:L2growth}
\norm{e^{tL}}_{2\to2}\le C_\epsilon e^{(a+\epsilon)t}.
\end{equation}
Neither reality nor semisimplicity of the rightmost eigenvalue is required.
To see explicitly why the supremum is attained, choose a positive number $w$ smaller than the real part of the unstable eigenvalue already obtained. The spectral set with real part at least $w$ is nonempty and finite. Its largest real part dominates all remaining spectral points, and therefore equals $s(L)$. Possible Jordan blocks on this finite-dimensional spectral subspace introduce polynomial factors in time; each such factor is bounded by a constant times $e^{\epsilon t}$. The allowance of an arbitrary positive $\epsilon$ in \eqref{eq:L2growth} accommodates this possibility without adding a semisimplicity hypothesis.

For short-time smoothing, the explicit heat formula gives, for every integer $m\ge0$,
\[
\norm{e^{tA}}_{H^m\to H^m}\le C_m,
\qquad\norm{e^{tA}}_{H^m\to H^{m+1}}\le C_mt^{-1/2},\qquad0<t\le1.
\]
Here the heat time is comparable to $t$ and the dilation factors are bounded. Since $\M_{\overline Q}:H^{m+1}\to H^m$ is bounded, solve
\[
w(t)=e^{tA}w_0+\int_0^t e^{(t-s)A}\M_{\overline Q}w(s)\dd s
\]
in the norm
\[
\sup_{0<t\le t_0}\left(\norm{w(t)}_{H^m}+\sqrt t\norm{w(t)}_{H^{m+1}}\right).
\]
For the component without a derivative gain, the integral is bounded by $C\int_0^t s^{-1/2}\dd s=2C\sqrt t$. For the component with a derivative gain, the required bound is
\[
\sqrt t\int_0^t(t-s)^{-1/2}s^{-1/2}\dd s=\pi\sqrt t.
\]
Thus the integral operator has norm at most $C_m\sqrt{t_0}$. Taking $t_0$ small gives boundedness on $H^m$ and gain of one derivative, with constants independent of the initial datum. These solutions coincide with the $L^2$ semigroup solution by uniqueness. Splitting the time interval into finitely many equal parts and using the semigroup property gives every integer derivative gain with its corresponding power of $t^{-1/2}$; interpolation gives all real orders $s_2\ge s_1\ge0$.
For example, a gain of $N$ derivatives is obtained by composing $N$ one-derivative estimates at time $t/N$, which yields $C_{m,N}t^{-N/2}$ for $0<t\le t_0$. Bounds on the remaining compact interval of positive times follow by another finite subdivision. Interpolation is legitimate on the divergence-free spaces because the projection $\PP$ is bounded on every Sobolev space and hence realizes them as complemented subspaces of the usual Sobolev scale. The estimates therefore apply to arbitrary real Sobolev indices used below, not merely to integer derivatives.

For large times, first use \eqref{eq:L2growth} and then apply a fixed positive interval of smoothing, using $H^{s_1}\hookrightarrow L^2$. Combining the short- and long-time bounds proves \eqref{eq:main-smoothing}. Finally, the eigenvector identity $\eta=e^{-\lambda}e^L\eta$ yields membership in every Sobolev space.
\end{proof}
\section{An unstable trajectory and non-uniqueness of Leray solutions}
\subsection{The common force and the nonlinear equation}

Let $\overline Q$ be the fixed background constructed in Section 4. For divergence-free pairs $V=(u,b)$ and $W=(v,h)$, define the bilinear operator
\begin{equation}\label{eq:B}
\B(V,W)=\PP\begin{pmatrix}u\cdot\nabla v-b\cdot\nabla h\\
u\cdot\nabla h-b\cdot\nabla v\end{pmatrix}.
\end{equation}
Then $\M_{\overline Q}w=-\B(\overline Q,w)-\B(w,\overline Q)$.
For the magnetic equation it is useful to distinguish the individual bilinear terms from their sum. If $[u,b]=u\cdot\nabla b-b\cdot\nabla u$, then
\begin{equation}\label{s5-bracket-div}
\diver[u,b]=u\cdot\nabla\diver b-b\cdot\nabla\diver u=0
\end{equation}
for divergence-free fields. Thus the magnetic components of $\B(V,V)$ and of $\B(\overline Q,w)+\B(w,\overline Q)$ are unchanged by projection; for the latter, the unprojected expression is $[\overline U,h]+[v,\overline B]$. An individual mixed term need not have this property. The product-space projection is convenient for estimates, while these identities retain the original magnetic equation.

Introduce the similarity variables
\begin{equation}\label{eq:ss-vars}
\xi=x/\sqrt t,\qquad \tau=\log t,\qquad
(u,b)(x,t)=t^{-1/2}\mathcal Q(\xi,\tau).
\end{equation}
Let $A=\Delta+\frac12(1+\xi\cdot\nabla)$ and prescribe the real-valued, smooth, compactly supported force profiles
\begin{equation}\label{eq:FG}
F=-A\overline U+\overline U\cdot\nabla\overline U-\overline B\cdot\nabla\overline B,
\qquad
G=-A\overline B+\overline U\cdot\nabla\overline B-\overline B\cdot\nabla\overline U.
\end{equation}
Thus $\overline Q$ is a steady solution of the equations in similarity variables; in fact, its total pressure may be taken to be zero. Writing $\mathcal Q=\overline Q+w$, the projected perturbation equation becomes
\begin{equation}\label{eq:nonlinear}
\partial_\tau w=Lw-\B(w,w).
\end{equation}
Once the background has been chosen, the force is fixed and does not depend on $w$. The Lie bracket of two divergence-free fields is divergence-free, so $\diver G=0$. Consequently, no additional pressure is needed in the magnetic equation when we return to physical variables.

\subsection{A second trajectory converging to the background in the remote past}

Choose an eigenvalue $\lambda=a+i\omega$, with $a>0$, attaining the spectral bound in Proposition \ref{prop:semigroup}, and a corresponding eigenvector $\eta$. Set
\begin{equation}\label{eq:qlinear}
q(\tau)=\Rea(e^{\lambda\tau}\eta).
\end{equation}
If $\omega=0$, choose a nonzero real eigenvector. If $\omega\ne0$, then $\Rea\eta$ and $\operatorname{Im}\eta$ are linearly independent over the reals. Otherwise, multiplication by a suitable complex phase would make $\eta$ a nonzero real vector, which cannot be an eigenvector of a real operator with a nonreal eigenvalue. Hence, for all $\tau$,
\begin{equation}\label{eq:q-bounds}
c e^{a\tau}\le\norm{q(\tau)}_2,\qquad
\norm{q(\tau)}_{H^m}\le C_m e^{a\tau}\quad(m\ge0).
\end{equation}
In the nonreal case, the lower bound follows from the strictly positive minimum on the phase circle of the continuous function
$\norm{\cos\theta\Rea\eta-\sin\theta\operatorname{Im}\eta}_2$.
In particular, this minimum is independent of $\tau$, even when the mode oscillates infinitely often as $\tau\to-\infty$. Since $L$ is real, conjugation commutes with its semigroup, and $q$ satisfies $\partial_\tau q=Lq$. Only this one eigenvector is used; possible generalized eigenvectors at the same spectral bound do not enter the prescribed leading term.

We now construct $w=q+r$, where the remainder decays faster than the linear mode.
\begin{proposition}\label{prop:trajectory}
There exist $T_0<0$ and a continuous, real-valued, divergence-free function $r:(-\infty,T_0]\to H^3$ such that
\begin{equation}\label{eq:r-int}
r(\tau)=-\int_{-\infty}^\tau e^{(\tau-s)L}
\B(q(s)+r(s),q(s)+r(s))\dd s,
\end{equation}
and, for every integer $m\ge0$, there is a constant $C_m$ such that
\begin{equation}\label{eq:r-decay}
\norm{r(\tau)}_{H^m}\le C_m e^{2a\tau},\qquad\tau\le T_0.
\end{equation}
Consequently, $\mathcal Q_1=\overline Q$ and $\mathcal Q_2=\overline Q+q+r$ are distinct smooth solutions with the same force in similarity variables.
\end{proposition}
\begin{proof}
Consider the real-valued closed subspace of the Banach space
\[
X=\left\{r\in C(( -\infty,T_0];H^3\cap\HH):
\norm r_X=\sup_{\tau\le T_0}e^{-2a\tau}\norm{r(\tau)}_{H^3}<\infty\right\}.
\]
Sobolev product estimates and boundedness of $\PP$ on $H^2$ give
\begin{equation}\label{eq:product}
\norm{\B(V,W)}_{H^2}\le C\norm V_{H^3}\norm W_{H^3}.
\end{equation}
Indeed, after applying at most two derivatives to a typical term $u\cdot\nabla v$, the extreme terms are controlled by $H^3\hookrightarrow W^{1,\infty}$, and the intermediate products by the same Sobolev embeddings. The constant is independent of the past time interval. Completeness of $X$ follows by applying uniform convergence to the weighted functions $e^{-2a\tau}r(\tau)$, and divergence-free, real-valued functions form a closed subspace.
Fix $0<\epsilon<a$. By \eqref{eq:main-smoothing}, for $s<\tau$,
\[
\norm{e^{(\tau-s)L}}_{H^2\to H^3}
\le C(1+(\tau-s)^{-1/2})e^{(a+\epsilon)(\tau-s)}.
\]
Define
\[
I_\gamma=\int_0^\infty(1+t^{-1/2})e^{-(\gamma-a-\epsilon)t}\dd t<\infty,
\qquad\gamma=2a,3a,4a.
\]
Writing $d_\gamma=\gamma-a-\epsilon$, these constants equal $d_\gamma^{-1}+\sqrt\pi\,d_\gamma^{-1/2}$. In particular, their finiteness uses the strict gap $a-\epsilon>0$ and does not impose any further relation between the real and imaginary parts of $\lambda$.
Denote the right-hand side of \eqref{eq:r-int} by $\Phi(r)$ and suppose that $\norm r_X\le R_0$. Estimating $\B(q,q)$, the two mixed terms, and $\B(r,r)$ separately, and substituting $t=\tau-s$, we obtain
\begin{equation}\label{eq:ball}
\norm{\Phi(r)}_X\le C_0+C_1R_0 e^{aT_0}+C_2R_0^2e^{2aT_0}.
\end{equation}
The first term uses $e^{2as}=e^{2a\tau}e^{-2at}$, while the other two use $e^{3as}$ and $e^{4as}$, respectively. For $r_1,r_2$ in this ball, bilinearity gives
\begin{equation}\label{eq:contraction}
\norm{\Phi(r_1)-\Phi(r_2)}_X
\le C\bigl(e^{aT_0}+R_0e^{2aT_0}\bigr)\norm{r_1-r_2}_X.
\end{equation}
First choose $R_0>2C_0$, and then take $T_0$ sufficiently negative that the right-hand side of \eqref{eq:ball} is at most $R_0$ and the coefficient in \eqref{eq:contraction} is less than $1/2$. The Banach fixed-point theorem gives a unique remainder and proves \eqref{eq:r-decay} for $m=3$. The time integral converges at the past endpoint because $a-\epsilon>0$, and at the upper endpoint because $t^{-1/2}$ is integrable; it also gives continuity. Both the semigroup and the bilinear map preserve real-valuedness.
More explicitly, the quadratic difference is
\[
\B(q+r_1,r_1-r_2)+\B(r_1-r_2,q+r_2).
\]
The factors containing $q$ give the first small coefficient in \eqref{eq:contraction}, and the remaining factors give the second. Thus the fixed-point uniqueness asserted here concerns the chosen weighted ball and prescribed linear mode; no uniqueness assertion for all ancient solutions is needed. Continuity of the integral follows by first removing short neighborhoods of its two endpoints and then using strong continuity of the semigroup on the remaining compact interval. The endpoint estimates are uniform on compact sets of $\tau$.

To obtain both higher regularity and the same decay rate, suppose that \eqref{eq:r-decay} holds for some real $m\ge3$. Since $H^{m-1}(\R^3)$ is an algebra,
\[
\norm{\B(q+r,q+r)}_{H^{m-1}}\le C_m e^{2a\tau},\qquad\tau\le T_0.
\]
Apply smoothing from $H^{m-1}$ to $H^{m+1/2}$ in \eqref{eq:r-int}. The time kernel is $1+t^{-3/4}$ and remains integrable, so
\[
\norm{r(\tau)}_{H^{m+1/2}}
\le C_m e^{2a\tau}\int_0^\infty(1+t^{-3/4})e^{-(a-\epsilon)t}\dd t.
\]
Starting at $m=3$ and gaining half a derivative at each step gives every Sobolev order; lower orders follow by the continuous embeddings. The same dominated-convergence argument gives continuity in each of these spaces. The upper endpoint remains integrable at every step because each application gains only three halves of a derivative from the nonlinear source. No bilinear estimate closed in $L^2$ is required.

Splitting \eqref{eq:r-int} at any finite initial time gives the finite-time mild equation for $q+r$. To justify time smoothness in the presence of the unbounded dilation coefficient, fix a compact interval of positive physical times and make the inverse transformation \eqref{eq:ss-vars}. The mild equation, and hence its distributional form, transforms to
\begin{equation}\label{s5-physical-projected}
\partial_t V=\Delta V-\B(V,V)+\PP(f,g).
\end{equation}
On that interval $V$ is continuous in every $H^m$, and the right-hand side is continuous in $H^{m-2}$. Therefore $V$ is continuously differentiable in $H^{m-2}$. Since $m$ is arbitrary, differentiating this equation repeatedly proves time smoothness in every fixed Sobolev space. The inverse coordinate change then gives local smoothness in $(\xi,\tau)$. This argument does not infer a global $L^2$ bound on $\xi\cdot\nabla w$ merely from unweighted spatial Sobolev regularity.

Thus $q+r$ satisfies the perturbation equation \eqref{eq:nonlinear}. By the lower bound for the linear mode and the decay estimate for the remainder,
\[
\norm{\mathcal Q_2(\tau)-\mathcal Q_1}_2\ge c e^{a\tau}-C e^{2a\tau}>0
\]
for all sufficiently negative $\tau$. Decrease $T_0$ if necessary.
\end{proof}

\subsection{Zero initial data and the energy equality}

\begin{proof}[Proof of Theorem \ref{thm:main}]
Take the two trajectories in Proposition \ref{prop:trajectory}, set $T=e^{T_0}$, and define
\begin{equation}\label{eq:two-solutions}
(u_j,b_j)(x,t)=t^{-1/2}\mathcal Q_j(x/\sqrt t,\log t),\qquad0<t\le T,
\end{equation}
with the common force in \eqref{eq:force} prescribed by \eqref{eq:FG}. Transforming \eqref{eq:nonlinear} directly shows that both pairs satisfy the projected system \eqref{eq:mhd}. The momentum pressure is recovered by the Helmholtz decomposition. Both the nonlinear term $u\cdot\nabla b-b\cdot\nabla u$ and $g$ in the magnetic equation are divergence-free. Its projected form therefore agrees with the original equation, without a magnetic pressure.
For later use, choose the pressure normalization
\begin{equation}\label{s5-pressure}
p_j=\sum_{i,l=1}^3\mathcal R_i\mathcal R_l
 (u_{j,i}u_{j,l}-b_{j,i}b_{j,l})+\Delta^{-1}\diver f,
\end{equation}
where $\mathcal R_i$ are the Riesz transforms. The last term belongs to $L^2$: for a smooth compactly supported force, its Fourier multiplier has size $O(|\xi|^{-1})$ near zero, which is square-integrable in dimension three, and the Fourier transform decays rapidly at infinity. Products are in $L^2$ by the Sobolev bounds. Consequently $p_j$ is uniformly bounded in $L^2$ on each compact interval of positive times.

The two trajectories are uniformly bounded in every fixed Sobolev norm for $\tau\le T_0$. The change of variables $x=\sqrt t\,\xi$ gives
\begin{equation}\label{eq:physical-est}
\norm{(u_j,b_j)(t)}_2^2\le Ct^{1/2},\qquad
\norm{\nabla(u_j,b_j)(t)}_2^2\le Ct^{-1/2},\qquad
\norm{(f,g)(t)}_2=t^{-3/4}\norm{(F,G)}_2.
\end{equation}
Thus the force belongs to $L^1_tL^2_x$, and both solutions attain zero initial data strongly in $L^2$, belong to the stated energy class, and extend continuously to $t=0$. Smoothness for positive times follows from Proposition \ref{prop:trajectory}. The nonlinear tensors are locally integrable. Together with convergence to the initial data and \eqref{eq:physical-est}, this allows the temporal test functions in the equations to extend to $t=0$, yielding the usual distributional formulation with initial data.
For completeness, the nonlinear tensors have spatial $L^1$ norm $O(t^{1/2})$. Moreover, \eqref{s5-pressure} and scaling give $\norm{p_j(t)}_2\le Ct^{-1/4}$, because $\norm{u_j\otimes u_j}_2+\norm{b_j\otimes b_j}_2\le Ct^{-1/4}$ and the force pressure has the same scaling. Thus the pressure is also integrable at zero. Integrating the weak equations from $\epsilon$ and taking $\epsilon\downarrow0$, the boundary pairing vanishes by $\norm{V_j(\epsilon)}_2\le C\epsilon^{1/4}$; all interior terms converge by these bounds, the force estimate, and Cauchy--Schwarz for diffusion.

For $0<\epsilon<t\le T$, multiply the two equations by $u_j$ and $b_j$, respectively, and integrate. The divergence-free conditions give
\[
\int(u_j\cdot\nabla u_j)\cdot u_j=0,\qquad
\int(u_j\cdot\nabla b_j)\cdot b_j=0,
\]
while the coupling terms satisfy
\[
\int(b_j\cdot\nabla b_j)\cdot u_j+
\int(b_j\cdot\nabla u_j)\cdot b_j=0.
\]
Integration by parts at spatial infinity is justified by inserting smooth compactly supported cutoffs $\chi_N(x)=\chi_1(x/N)$, equal to one on the ball of radius $N$. On $[\epsilon,t]$, the cubic boundary terms are bounded by $CN^{-1}\norm{V_j}_\infty\norm{V_j}_2^2$, and the pressure term by $CN^{-1}\norm{p_j}_2\norm{u_j}_2$. The diffusion errors tend to zero by the $L^2$ bounds on $V_j$ and its gradient. These estimates are integrable on this positive-time interval, so the cutoffs may be removed. We thus obtain the MHD energy equality on $[\epsilon,t]$. Moreover, \eqref{eq:physical-est} gives
\[
\left|\int_{\R^3}(f\cdot u_j+g\cdot b_j)\dd x\right|
\le Ct^{-3/4}t^{1/4}=Ct^{-1/2},
\]
which is integrable at the origin. As $\epsilon\downarrow0$, the initial energy tends to zero, and the dissipation and force-pairing integrals converge. Hence
\begin{equation}\label{eq:energy-equality}
\frac12\norm{(u_j,b_j)(t)}_2^2+
\int_0^t\norm{\nabla(u_j,b_j)(s)}_2^2\dd s
=\int_0^t\!\int_{\R^3}(f\cdot u_j+g\cdot b_j)\dd x\dd s.
\end{equation}
Both solutions are therefore Leray energy solutions. They are distinct by the strict separation of the trajectories and \eqref{eq:two-solutions}.

Finally, the initial profile $b_R$ is nonzero, and both the toroidal construction and rescaling preserve $\overline B\not\equiv0$. Identity \eqref{eq:induction-zero} remains valid after rescaling, so $G=-A\overline B$. If $G=0$, then $e^{tA}\overline B=\overline B$, whereas the decay $e^{-t/4}$ of the free semigroup forces $\overline B=0$, a contradiction. Thus the magnetic force is also nonzero.
Equivalently, compact support permits the direct calculation
\begin{equation}\label{s5-magnetic-work}
\ip{G}{\overline B}
=\norm{\nabla\overline B}_2^2+\tfrac14\norm{\overline B}_2^2>0.
\end{equation}
This also checks the sign of the dilation term in the force profile and identifies a strictly positive magnetic contribution to the background energy balance.
\end{proof}

\begin{remark}
The conclusion allows forcing in both equations and relies on the existence of a specific family of profiles. The proof does not claim that every nonconstant magnetic profile is unstable, or that direct multiplication by an axial cutoff preserves eigenvalues. Our background is supported in a compact toroidal region. Its instability is first established at a fixed positive viscosity and then transferred to the operator in similarity variables.
\end{remark}
\appendix
\section{A radial negative-direction criterion and localization of magnetic flux}
\subsection{An independent variational check}

This appendix gives a direct variational derivation from \eqref{eq:components}. It introduces no additional spectral assumption into the main proof and does not replace the tracking of the prescribed seed eigenvalue in Theorem \ref{thm:radial}. Its purpose is to check the signs independently and to show directly that a radially compactly supported magnetic field may have an unstable eigenvalue.

Let the real eigenvalue parameter be $\lambda>0$, set $\mu=\lambda^2$, and write $w=v_r$ and $D_rw=w'+w/r$. The divergence-free constraints and the magnetic equation give
\begin{equation}\label{eq:reconstruct}
\begin{gathered}
v_z=\frac{i}{k}D_rw,\qquad h_r=\frac{ikb}{\lambda}w,\qquad
h_z=\frac{ikb}{\lambda}v_z-\frac{b'}\lambda w,\\
v_\theta=-\left(\frac{r\Omega'}\lambda+\frac{2\lambda\Omega}{\lambda^2+k^2b^2}\right)w,
\qquad h_\theta=\frac{ikb}{\lambda}v_\theta+\frac{r\Omega'}\lambda h_r.
\end{gathered}
\end{equation}
In the axial momentum equation, $ikb(-b'w/\lambda)$ and $b'h_r$ cancel exactly, so
\[
q=-\frac{\mu+k^2b^2}{k^2\lambda}D_rw.
\]
Substituting its derivative into the radial momentum equation gives
\begin{equation}\label{eq:pencil-ode}
\bigl((\mu+k^2b^2)D_rw\bigr)'
=k^2\left(\mu+k^2b^2+2r\Omega\Omega'
+\frac{4\mu\Omega^2}{\mu+k^2b^2}\right)w.
\end{equation}
Let $V_k$ be the completion of $C_c^\infty((0,\infty))$ under the norm
\[
\norm w_{V_k}^2=\int_0^\infty\left(|w|^2+k^{-2}|D_rw|^2\right)r\dd r.
\]

For compactly supported test functions, integration of the cross term gives
\begin{equation}\label{eq:app-vector-energy}
\int_0^\infty|D_rw|^2r\dd r
=\int_0^\infty\left(|w'|^2+\frac{|w|^2}{r^2}\right)r\dd r.
\end{equation}
The same identity extends to $V_k$ by completion. Thus the definition of $V_k$ controls the Cartesian energy of the radial vector component near the axis, and compactness on annuli follows from a genuine $H^1$ bound. Multiplying \eqref{eq:pencil-ode} by $\overline\varphi r$ and integrating by parts gives the derivative term with a minus sign; moving it to the other side and dividing by $k^2$ gives the form below. This also fixes its sign without choosing a boundary value at the axis.

The self-adjoint weak formulation of \eqref{eq:pencil-ode} is $q_\mu(w,\varphi)=0$ for every $\varphi\in V_k$, where the quadratic form is
\begin{equation}\label{eq:pencil}
q_\mu[w]=\int_0^\infty\left[
(\mu+k^2b^2)\left(|w|^2+k^{-2}|D_rw|^2\right)
+\left(2r\Omega\Omega'+\frac{4\mu\Omega^2}{\mu+k^2b^2}\right)|w|^2
\right]r\dd r.
\end{equation}
The full weak equation is required here: the identity $q_\mu[w]=0$ for a single test function does not by itself yield an eigenvector.

\begin{proposition}\label{prop:negative}
Let $\Omega$ be as in Section 2, and suppose that $b$ satisfies \eqref{eq:bclass} and is nonzero in a neighborhood of $\supp\Omega$. If there exists $w_0\in C_c^\infty((0,\infty))$, supported in the region where $b\ne0$, such that
\begin{equation}\label{eq:negative}
\int_0^\infty\left[k^2b^2\left(|w_0|^2+k^{-2}|D_rw_0|^2\right)
+2r\Omega\Omega'|w_0|^2\right]r\dd r<0,
\end{equation}
then $M_{k,b}$ has a positive real eigenvalue. The profile $b$ is allowed to vanish identically outside a sufficiently large disk.
\end{proposition}
\begin{proof}
Since $b\ne0$ on the support of the test function, \eqref{eq:negative} implies $q_\mu[w_0]<0$ for small positive $\mu$. For large $\mu$, we have $q_\mu\ge(\mu-C)\norm{\cdot}_{V_k}^2>0$. On each compact interval of positive $\mu$, the bounded self-adjoint operator representing the form is norm-continuous in $\mu$ and satisfies
\[
q_{\mu_2}[w]-q_{\mu_1}[w]\ge(\mu_2-\mu_1)\norm w_{V_k}^2,
\qquad\mu_2>\mu_1>0,
\]
because $\mu/(\mu+k^2b^2)$ is nondecreasing in $\mu$.

The part of the form involving $\Omega$ and $\Omega'$ is compact on $V_k$: its coefficients are supported in a fixed annulus, on which bounded subsets of $V_k$ have one-dimensional $H^1$ bounds. The remaining part is bounded below by $\mu I$, so the essential spectrum of the representing operator lies in $[\mu,\infty)$. More precisely, represent $q_\mu$ by a bounded self-adjoint operator $H(\mu)$ using the $V_k$ inner product. On each compact subinterval of $(0,\infty)$, the rational coefficient in \eqref{eq:pencil} and its $\mu$ derivative are uniformly bounded, so $H(\mu)$ is norm-continuous. Consequently $\inf\sigma(H(\mu))$ is continuous. The compact term has no effect on the lower bound $\mu$ for the essential spectrum. The negative test direction gives a strictly negative spectral bottom at one positive parameter, while the bound $q_\mu\ge(\mu-\|2r\Omega\Omega'\|_\infty)\|\cdot\|_{V_k}^2$ gives a positive bottom at a sufficiently large parameter. The bottom of the spectrum varies continuously with $\mu$ and changes from negative to positive. It therefore equals zero at some $\mu_*>0$. Since zero lies below the essential spectrum, it is an eigenvalue. Hence there exists $0\ne w\in V_k$ satisfying the full weak equation.

Set $\lambda=\sqrt{\mu_*}$ and reconstruct the remaining components using \eqref{eq:reconstruct}. All denominators are controlled by $\lambda>0$ and $\mu_*+k^2b^2\ge\mu_*$. Thus the reconstructed vector belongs to $\HH_k$, and the divergence-free condition for the magnetic field follows directly from the formulas. The pressure $q$ defined above belongs to radial $L^2$, and \eqref{eq:pencil-ode} shows that its weak derivative does as well. The momentum residual is therefore a legitimate pressure gradient and is annihilated by $P_k$.

In particular, the reconstructed poloidal magnetic components can be written as
\[
h_r=\frac{ik}{\lambda}bw,\qquad h_z=-\frac1\lambda D_r(bw),
\]
so their divergence cancels distributionally, including across every zero of $b$. The pressure derivative is the $L^2(r\dd r)$ function
\begin{equation}\label{eq:app-pressure-derivative}
q'=-\frac1\lambda\left(\mu_*+k^2b^2+2r\Omega\Omega'
+\frac{4\mu_*\Omega^2}{\mu_*+k^2b^2}\right)w.
\end{equation}
All coefficients here are bounded. Together with $q\in L^2(r\dd r)$, this places $(q',0,ikq)$ in the gradient subspace appearing in \eqref{eq:s2-pressure-projection}. Thus projection removes an actual energy-space gradient, not merely a formal pressure term.
 We obtain $M_{k,b}(v,h)=\lambda(v,h)$, with $v_r=w\ne0$.
\end{proof}

If $2r\Omega\Omega'<0$ on some annulus, choose a nonzero $w_0$ supported there and set $b=\varepsilon b_0$, where $b_0$ is a fixed smooth, positive, nonconstant profile. The first term in \eqref{eq:negative} is $O(\varepsilon^2)$, while the second is strictly negative. The criterion therefore holds for sufficiently small $\varepsilon$. Any radial cutoff that leaves this annulus unchanged preserves the negative test direction. Thus the periodic spectral input in the main theorem may either be taken from \cite{LWW} or verified directly for this specific family by the proposition above.

The smallness condition can be made explicit at the level of a fixed test direction. For the example \eqref{eq:example}, choose an annulus $I\Subset(2,3)$ and a nonzero real $w_0\in C_c^\infty(I)$. Write $b_0=1+d e^{-r^2}$ and define the positive constants
\begin{align}\label{eq:app-test-constants}
A_I&=-\int_I2r\Omega\Omega'|w_0|^2r\dd r,\qquad N_I=\|w_0\|_{V_k}^2,\notag\\
C_I&=\int_I b_0^2\left(k^2|w_0|^2+|D_rw_0|^2\right)r\dd r,\notag\\
D_I&=\frac4{k^2}\int_I\frac{\Omega^2}{b_0^2}|w_0|^2r\dd r.
\end{align}
For $b=\varepsilon b_0$ and $\mu>0$, the denominator in the last term of \eqref{eq:pencil} is at least $k^2\varepsilon^2b_0^2$. Consequently,
\begin{equation}\label{eq:app-quantitative-negative}
q_\mu[w_0]\le\varepsilon^2 C_I-A_I
+\mu\left(N_I+\frac{D_I}{\varepsilon^2}\right).
\end{equation}
In particular, the concrete choices
\[
0<\varepsilon^2\le\frac{A_I}{2C_I},\qquad
0<\mu\le\frac{A_I}{4(N_I+D_I/\varepsilon^2)}
\]
give $q_\mu[w_0]\le-A_I/4<0$. All constants are integrals of fixed smooth functions on one annulus; no limiting eigenfunction or numerical spectral calculation is needed to verify the input. At a sufficiently large positive $\mu$, the entire form is positive, so the spectral-crossing argument above produces a real unstable eigenvalue at an intermediate parameter.

The same estimate holds for every radial cutoff that leaves $I$ unchanged, since the form evaluated at $w_0$ uses only the coefficients on its support. If the cutoff also leaves a neighborhood of $\supp\Omega$ unchanged, all hypotheses of Proposition~\ref{prop:negative} remain valid. This proves existence of an unstable truncated column directly, whereas Theorem~\ref{thm:radial} additionally tracks a prescribed seed eigenvalue. Either route supplies a nonzero, nonconstant magnetic profile; the subsequent viscous persistence and whole-space bending still require the arguments of Sections 2 and 3.

\subsection{Why a direct axial cutoff requires a separate argument}

Closing the cylinder into a torus allows a purely toroidal magnetic field to have compact support without a return-flux correction. For axial localization in a straight cylinder, fix $\zeta\in C_c^\infty(\R)$ with $\zeta=1$ near zero. The field $b(r)\zeta(z/L)e_z$ has divergence $L^{-1}b(r)\zeta'(z/L)$, so a radial correction is required. If one prescribes $B_z=b(r)\chi(r/R)\zeta(z/L)$ and reconstructs $B_r$ by integrating a magnetic potential, then the potential at large radius may equal a nonconstant axial flux $M(z)$. This gives $B_r=-M'(z)/r$ and hence
\[
\int |B_r|^2r\dd r\dd z=\infty
\]
whenever $M'$ is nonzero. One should therefore cut off the magnetic potential rather than only the axial magnetic component.

Specifically, let $\Psi(r)=\int_0^r sb(s)\dd s$ and $F_R(r)=\chi(r/R)\Psi(r)$, and set
\[
\Psi_{R,L}(r,z)=\zeta(z/L)F_R(r),\qquad
B_r=-\frac{\zeta'(z/L)}{Lr}F_R(r),\qquad
B_z=\zeta(z/L)\frac{F_R'(r)}r.
\]
Since $F_R=O(r^2)$, this field is smooth at the axis, exactly divergence-free, and compactly supported.
The regularity at the axis follows from the radial smoothness assumed for $b(|y|)$, in addition to the order of vanishing: near zero, $F_R(r)/r^2$ is a smooth radial function, so $B_re_r$ is a smooth Cartesian field. The energies can be checked directly:
\begin{align}\label{eq:app-return-energy}
\|B_r\|_2^2&=\frac{2\pi}{L}\|\zeta'\|_{L^2(\R)}^2
\int_0^\infty\frac{|F_R|^2}{r}\dd r,\notag\\
\|B_z\|_2^2&=2\pi L\|\zeta\|_{L^2(\R)}^2
\int_0^\infty\frac{|F_R'|^2}{r}\dd r.
\end{align}
Both integrals are finite for each finite $R,L$: near the axis, $F_R=O(r^2)$ and $F_R'=O(r)$, and both functions vanish at large radius. The net axial magnetic flux through each cross-section is zero, since
\[
2\pi\int_0^\infty B_z(r,z)\,r\dd r
=2\pi\zeta(z/L)\bigl(F_R(\infty)-F_R(0)\bigr)=0.
\]

Its radial magnetic profile
\[
\widetilde b_R=\frac{F_R'}r=\chi(r/R)b(r)+\frac{\chi'(r/R)}{Rr}\Psi(r)
\]
includes return flux and satisfies $\int_0^\infty r\widetilde b_R\dd r=0$. In general it differs from the profile $b_R$ used in the main text, but satisfies
\[
\sup_R\norm{\widetilde b_R}_\infty<\infty,\quad
\widetilde b_R\to b\text{ pointwise},\quad
\norm{\widetilde b_R'-b'}_\infty
\le\sup_{r\ge R}|b'(r)|+CR^{-1}\norm b_\infty\longrightarrow0.
\]
The last estimate follows by differentiation and the bound $|\Psi(r)|\le\norm b_\infty r^2/2$. In fact,
\[
\widetilde b_R'-b'
=(\chi(r/R)-1)b'+\frac{2\chi'(r/R)}R b
+\frac{\chi''(r/R)}{R^2r}\Psi
-\frac{\chi'(r/R)}{Rr^2}\Psi.
\]
The last three terms are supported where $R\le r\le2R$ and are bounded by $C R^{-1}\|b\|_\infty$. This yields exactly the norm estimate used above, rather than convergence of the full multiplication operators in operator norm. The proof of Theorem~\ref{thm:radial} therefore applies to this radial return-flux profile as well. Nevertheless, the subsequent spectral transfer to the actual whole-space operator after insertion of an axial cutoff does not follow directly from the fixed-frequency result. The main text completes the non-uniqueness argument using the toroidal construction proved above.

\bigskip

\begin{center}
\textbf{Acknowledgement}
\end{center}

The research of Xu was supported by National Natural Science Foundation of China (No. 12401134). The research of Zhang was
supported by National Natural Science Foundation of China (No. 12301133) and the Postdoctoral Science
Foundation of China (No. 2024T170353).

\end{document}